\UseRawInputEncoding
\documentclass{amsart}

\usepackage[T1]{fontenc} 
\usepackage{lmodern,dsfont} 
\usepackage[english]{babel} 
\usepackage{amsmath,amsthm,graphicx,float,mathtools,tikz-cd,amssymb,hyperref} 
\newcommand{\mb}[1]{\mathbb{#1}}

\newcommand{\inj}{\hookrightarrow}

\newcommand{\id}{\text{id}}
\newcommand{\ip}[1]{\langle#1\rangle}

\newcommand{\Hom}{\text{Hom}}

\newcommand{\st}{\text{ : }}

\newcommand{\mc}[1]{\mathcal{#1}}

\newcommand{\G}{\Gamma}

\newcommand{\Q}{\mathbb{Q}}

\newcommand{\R}{\mathbb{R}}

\newcommand{\Z}{\mathbb{Z}}

\newcommand{\F}{\mathbb{F}}

\newcommand{\btz}{\begin{tikzcd}}
\newcommand{\etz}{\end{tikzcd}}
\newcommand{\x}[1]{\xrightarrow{#1}}

\newcommand{\bpm}{\begin{pmatrix}}
\newcommand{\epm}{\end{pmatrix}}
\newcommand{\bg}{\bigg}

\newcommand{\Spec}{\text{Spec}}

\newcommand{\smsh}{\wedge}
\newcommand{\tx}[1]{\text{#1}}
\newcommand{\tb}[1]{\textbf{#1}}
\newcommand{\ti}[1]{\textit{#1}}

\newcommand{\ov}[1]{\overline{#1}}

\newcommand{\aln}[1]{\begin{align*}#1\end{align*}}
\newcommand{\un}[1]{\underline{#1}}
\DeclareMathOperator{\colim}{\text{colim}}

\usepackage{graphicx}

\newtheorem{theorem}{Theorem}[section]
\newtheorem{lemma}[theorem]{Lemma}

\theoremstyle{definition}
\newtheorem{definition}[theorem]{Definition}
\newtheorem{example}[theorem]{Example}

\theoremstyle{remark}
\newtheorem{remark}[theorem]{Remark}

\theoremstyle{construction}
\newtheorem{construction}[theorem]{Construction}

\theoremstyle{corollary}
\newtheorem{corollary}[theorem]{Corollary}

\theoremstyle{proposition}
\newtheorem{proposition}[theorem]{Proposition}

\numberwithin{equation}{section}

\begin{document}
\title{Smashing Localizations in Equivariant Stable Homotopy}

\author{Christian Carrick}
\address{University of California, Los Angeles, Los Angeles, CA 90095}
\email{carrick@math.ucla.edu}






\begin{abstract}
We study how smashing Bousfield localizations behave under various equivariant functors. We show that the analogs of the smash product and chromatic convergence theorems for the Real Johnson-Wilson theories $E_{\mathbb{R}}(n)$ hold only after Borel completion. We establish analogous results for the $C_{2^n}$-equivariant Johnson-Wilson theories constructed by Beaudry, Hill, Shi, and Zeng. We show that induced localizations upgrade the available norms for an $N_\infty$-algebra, and we determine which new norms appear. Finally, we explore generalizations of our results on smashing localizations in the context of a quasi-Galois extension of $E_\infty$-rings.
\end{abstract}


\maketitle

\section{Introduction}
Chromatic homotopy theory studies the robust connection between stable homotopy theory and the theory of formal groups. This connection comes from a theorem of Quillen \cite{Quillen}, which gives a refinement of complex cobordism homology $MU_*(-)$ to a functor
\[\mc F:\tx{Spectra}\to\tx{QCoh}(\mc M_{FG})\]
where the latter is the category of quasi-coherent sheaves on $\mc M_{FG}$, the moduli stack of 1-dimensional, commutative formal groups \cite{Naumann}\cite{Goerss}. This functor retains a surprising amount of information about the stable homotopy category. In particular, we may localize at a prime $p$, and $(\mc M_{FG})_{(p)}$ carries a natural filtration by height. The height filtration rigidifies $(\mc M_{FG})_{(p)}$, and a series of conjectures made by Ravenel in \cite{Rav84}, proven in \cite{DHS}\cite{HS}\cite{Ravenel}, establish that this rigidity is reflected quite strongly in the stable homotopy category. We explore some of these conjectures in an equivariant context.


In $\tx{QCoh}((\mc M_{FG})_{(p)})$, there is a natural localization functor $\iota_n^*$ given by restricting sheaves to the open substack $\mc M_{FG}^{\le n}$ of formal groups of height $\le n$, so that
\[\iota_n^*(-)=\mc O_{\mc M_{FG}^{\le n}}\otimes_{\mc O_{\mc M_{FG}}}(-)\]
Bousfield localization of spectra provides a way to, in some sense, lift this localization functor along $\mc F$. The Johnson-Wilson theories, $E(n)$, have the property that, for a spectrum $X$, $L_{E(n)}(X)=0$ if and only if $\iota_n^*(\mc F(X))=0$, and the smash product theorem states that $L_{E(n)}(-)=L_{E(n)}(S^0)\smsh (-)$. The chromatic convergence theorem, says, moreover, that a p-local finite spectrum $X$ is determined by these localizations: the natural map
\[X\to\varprojlim\bg(\cdots\to L_{E(n)}(X)\to L_{E(n-1)}(X)\to\cdots\to L_{E(0)}(X)\bg)\]
is an equivalence, reflecting the fact that finitely presented sheaves in $\tx{QCoh}((\mc M_{FG})_{(p)})$ are determined by their restrictions to each open substack in the height filtration.

$MU$ carries a natural action of $C_2$, the cyclic group of order two, coming from complex conjugation. The resulting $C_2$-spectrum, $MU_\R$, also known as Real bordism theory, was defined by Fujii \cite{Fujii} and Landweber \cite{Landweber} and studied extensively by Hu and Kriz \cite{Hu}, who established a similar connection between genuine $C_2$-spectra and the theory of formal groups. In particular, they define $C_2$-equivariant lifts $E_\R(n)$ of the $E(n)$'s. We establish that the analogous theorems do not hold in genuine $C_2$-spectra for the $E_\R(n)$'s:
\begin{theorem}\label{1.1}
For $n>0$, $L_{E_\R(n)}(-)$ is not smashing. In fact, for $X\in Sp^{C_2}$, 
\[L_{E_\R(n)}(X)\simeq F(E{C_2}_+,L_{E_\R(n)}(S^0)\smsh X)\simeq F(E{C_2}_+,i_*L_{E(n)}(S^0)\smsh X)\]
\end{theorem}
\begin{theorem}\label{1.2}
If $X$ is a $2$-local finite $C_{2}$-spectrum, we have a diagram
\[
\btz
X\arrow[r]\arrow[dr]&F(E{C_{2}}_+,X)\arrow[d,"\simeq"]\\
&\varprojlim_n L_{E_\R(n)}(X)
\etz
\]
\end{theorem}
There is, however, another perspective on these theorems: the chromatic convergence and smash product theorems for the $E_\R(n)$'s do not hold in genuine $C_2$-spectra, but they do hold in Borel $C_2$-spectra. That we need to pass to Borel $C_2$-spectra is perhaps unsurprising because $MU_\R$ itself is Borel complete: it is a theorem of Hu and Kriz that the map
\[MU_\R\to F(E{C_2}_+,MU_\R)\]
is an equivalence in $Sp^{C_2}$, and similarly for $E_\R(n)$. In the case of the nilpotence and thick subcategory theorems, the analogs for Real bordism theory are easily seen to fail in genuine $C_2$-spectra. Passing to Borel $C_2$-spectra, the nilpotence conjecture still fails for $MU_\R$, but the analog of the thick subcategory theorem is more delicate, and we remark on the difficulties in \hyperref[4.3]{4.3}.

In their solution to the Kervaire Invariant One problem \cite{HHR}, Hill, Hopkins, and Ravenel construct genuine $C_{2^n}$-spectra $MU^{((C_{2^n}))}:=N_{C_2}^{C_{2^n}}MU_\R$ that bring Real bordism theory into the $C_{2^n}$-equivariant context. These play an essential role in their proof: their detecting spectrum is a localization of $MU^{((C_8))}$. Recently, Beaudry, Hill, Shi, and Zeng have constructed versions of Johnson-Wilson theories in this context \cite{BHSZ}, which they call $D^{-1}BP^{((G))}\ip{m}$. We give a description of the Bousfield classes of these spectra (\hyperref[4.6]{4.6}) and deduce analogous results:
\begin{theorem}\label{1.3}
Let $E_G(m)$ denote the $C_{2^n}$-spectrum $D^{-1}BP^{((G))}\ip{m}$ constructed in \cite{BHSZ}, where $h=2^{n-1}m$.
\begin{itemize}
\item If $m>0$, then $E_G(m)$ is not smashing. Moreover, for $X\in Sp^{C_{2^n}}$,
\[L_{E_G(m)}(X)\simeq F(E{C_{2^n}}_+,i_*L_{E(h)}(S^0)\smsh X)\]
\item If $X$ is a 2-local finite $C_{2^n}$-spectrum, we have a diagram
\[
\btz
X\arrow[r]\arrow[dr]&F(E{C_{2^n}}_+,X)\arrow[d,"\simeq"]\\
&\varprojlim_m L_{E_G(m)}(X)
\etz
\]
\end{itemize}
\end{theorem}

Our analysis begins with the observation that the $E_\R(n)$'s and $E_G(m)$'s are Bousfield equivalent to certain induced $G$-spectra. We therefore study in general how Bousfield classes in the equivariant context behave under various change of group functors, most of which send smashing Bousfield classes to smashing Bousfield classes (\hyperref[3.12]{3.12}). The exceptional case is that of the induction functor 
\[G_+\smsh_H(-):Sp^H\to Sp^G\]
for a subgroup $H\subset G$. We give a necessary and sufficient condition (\hyperref[3.18]{3.18}, \hyperref[3.19]{3.19}) for a smashing Bousfield class to be preserved by induction, and we find that the above formula for $L_{E_\R(n)}$ is generic in this context (\hyperref[3.20]{3.20}).

\subsection*{Summary} 
In Section \hyperref[sec2]{2}, we review Bousfield localization of $G$-spectra and the relationship between smashing localizations and tensor idempotents. In Section \hyperref[sec3]{3}, we study the interaction between Bousfield localization functors and change of group functors in general, specializing to smashing localizations in \hyperref[sec3.2]{3.2}. 

From here, we move to applications of Section  \hyperref[sec3]{3}, beginning in Section  \hyperref[sec4]{4} with the proofs of \hyperref[1.1]{1.1}, \hyperref[1.2]{1.2}, and \hyperref[1.3]{1.3}, and a remark on analogs of the nilpotence and thick subcategory theorems. Nonequivariantly, the functors $L_{E(n)}$ have the additional remarkable property that, the subcategories of finite $p$-local spectra
\[\mc C_{\ge n}=\{X\in Sp^\omega_{(p)}\st L_{E(n-1)}(X)=0\}\]
form a complete list of the thick tensor ideals in the category of finite $p$-local spectra. A description of the thick tensor ideals in finite $G$-spectra has been given for all $G$ abelian by \cite{BS}\cite{Barthel}, and for $n>0$, none of the thick tensor ideals in finite $C_2$-spectra correspond to $L_{E_\R(n)}$ in an analogous way. For $G=C_{p^n}$, we construct a family of new $G$-spectra $E(\mc J)$ - indexed by the thick tensor ideals $\mc J$ in $(Sp^G)^\omega_{(p)}$ - such that $L_{E(\mc J)}$ is smashing, $\mc J$ is the collection of finite acyclics of $E(\mc J)$, and the geometric fixed points of $E(\mc J)$ at any subgroup is a nonequivariant $E(n)$ (\hyperref[4.10]{4.10}). 

In Section \hyperref[sec5]{5}, we use formulae like the above for $L_{E_\R(n)}$ to observe that induced localizations upgrade the norms available in an $N_\infty$-algebra, and we determine exactly which new norms appear. This generalizes a result of Blumberg and Hill that if $E\in Sp^G$ is a cofree $E_\infty$-ring, it is automatically genuine $G$-$E_\infty$ \cite{BH}. 

Finally, in Section \hyperref[sec6]{6}, we return to the Borel perspective on the main theorems mentioned above. It is a result of \cite{BDS} (upgraded to the level of symmetric monoidal $\infty$-categories by \cite{MNN}) that $Sp$, the category of nonequivariant spectra, is equivalent to the category of modules in $Sp^{C_2}$ over the $E_\infty$-ring $A=F({C_2}_+,S^0)$, so that the coinduction functor becomes restriction of scalars, and the restriction functor becomes extension of scalars. Moreover, extension of scalars induces an equivalence between the category of Borel $C_2$-spectra and $(\tx{Mod}_{Sp^{C_2}}(A))^{hC_2}$.

We show that, by analogy, if $\eta:\mathds{1}\to A$ is a quasi-Galois extension in a symmetric monoidal stable $\infty$-category (\hyperref[6.4]{6.4}), it is often possible to use a norm construction to take a smashing $A$-module $M$ and produce a smashing object in the category of $A$-locals. This is equivalent to producing, as in \hyperref[1.1]{1.1}, a smashing-then-complete type localization formula for $\eta^*M$. We give a necessary and sufficient condition for this localization to be smashing in the category of $A$-modules (\hyperref[6.15]{6.15}).


\subsection*{Notation and Conventions} Unless explicitly stated otherwise, $G$ is a finite group. We will use the letters $H$ and $K$ to denote subgroups. $X,Y$, and $E$ will be used to denote a $G$-spectrum, and $Z$ will be used when referring to an acyclic. $Sp^G$ will denote the category of orthogonal $G$-spectra, and $(Sp^G)^\omega$ its subcategory of compact objects. We use the term ring spectrum to refer to a monoid in $Ho(Sp^G)$ and $[-,-]^G$ will denote morphisms in $Ho(Sp^G)$. If $H\subset G$ and $g\in G$, $^gH:=gHg^{-1}$.
\subsection*{Acknowledgments} We thank Mike Hill for suggesting the project and for his constant guidance and support. We would also like to thank Paul Balmer for many helpful conversations.


\section{Equivariant Bousfield Classes}\label{sec2}
In this section, we review what we need from equivariant Bousfield localization following \cite{Hill} and smashing localizations following \cite{Balmer}.
\subsection{Equivariant categories of acyclics}\label{sec2.1} We begin with a review of the characterization of acyclics in an equivariant context given in \cite{Hill}.
\begin{definition}\label{2.1}
If $E$ is a $G$-spectrum, we let $\mc{Z}_E$ denote the category of $E$-acyclics: the full subcategory of $Sp^G$ consisting of all $Z$ such that $E\smsh Z$ is equivariantly contractible. We let $\mc{L}_E$ denote the category of $E$-locals: the full subcategory of $Sp^G$ consisting of all $X$ such that $Sp^G(Z,X)\simeq*$ for all $Z\in\mc{Z}_E$. We say $E,F\in Sp^G$ are Bousfield equivalent (denoted $\ip{E}=\ip{F}$) if $\mc Z_E=\mc Z_F$.
\end{definition}
Since the geometric fixed point functors $\Phi^H$ are symmetric monoidal and jointly conservative, this gives us a concrete way to describe $\mc{Z}_E$:
\begin{proposition}\label{2.2}$($\cite{Hill}\textnormal{, Proposition 3.2}$)$ If $Z\in Sp^G$, then $Z\in\mc{Z}_E$ if and only if $\Phi^H(Z)\in \mc{Z}_{\Phi^H(E)}$ for all subgroups $H\subset G$:
\[\mc{Z}_E=\bigcap\limits_{H\subset G}(\Phi^H)^{-1}(\mc Z_{\Phi^H(E)})\]
\end{proposition}
\begin{corollary}\label{2.3}
\cite{Hill} Suppose $E\in Sp^{G}$ has the property that $\Phi^{H}(E)\simeq*$ for all $H\subset G$ nontrivial, then $\mc{Z}_E=(\Phi^{\{e\}})^{-1}(\mc{Z}_{\Phi^{\{e\}}(E)})$. That is, $Z\in Sp^{G}$ is $E$-acyclic if and only if its underlying spectrum is $\Phi^{\{e\}}(E)$-acyclic.
\end{corollary}
From this, we deduce a useful characterization of the Bousfield classes of the Real Johnson-Wilson theories introduced by Hu-Kriz \cite{Hu} and studied extensively by Kitchloo-Wilson \cite{Kitchloo}.
\begin{example}\label{2.4}
Let $E_\R(n)$ denote the $n$-th Real Johnson-Wilson theory, $E_{(k,\mb G)}$ the Lubin Tate theory associated to a perfect field $k$ of characteristic 2 and $\mb G$ a height $n$ formal group over $k$, regarded as a $C_2$-spectrum as in  \cite{Hahn}, and $E(n)$ the usual nonequivariant Johnson-Wilson theory. Then
\[\ip{E_\R(n)}=\ip{E_{(k,\mb G)}}=\ip{{C_2}_+\smsh E(n)}\]
\end{example}

\begin{proof}
These three $C_2$-spectra all have contractible geometric fixed points, and the Bousfield classes of their underlying spectra agree.
\end{proof} 

\subsection{Smashing spectra and idempotent triangles}\label{sec2.2} We review the theory of smashing localizations - for more details see \cite{Balmer},\cite{Lurie},\cite{Rav84}, and \cite{Ravenel}. We first recall the following basic fact about Bousfield localization that we will use repeatedly.
\begin{lemma}\label{2.5}
If $E\in Sp^G$ is a ring spectrum, then any module $M$ over $E$ (e.g. $E$ itself) is $E$-local.
\end{lemma}
\begin{proof}
Let $Z\in\mc Z_E$, then any map $f:Z\to M$ factors as follows
\[
\btz
Z\arrow[r,"f"]\arrow[d]&M\\
E\smsh Z\arrow[r,"1_E\smsh f"]&E\smsh M\arrow[u,"\mu_M"]
\etz
\]
but then $E\smsh Z\simeq*$, hence $f$ is null.
\end{proof}
\begin{definition}\label{2.6}
For $E\in Sp^G$, let $L_E$ denote the corresponding Bousfield localization functor. We say that $L_E$ is a smashing localization or that $E$ is a smashing $G$-spectrum if the natural map
\[L_E(S^0)\smsh X\to L_E(S^0)\smsh L_E(X)\to L_E(X)\]
is an equivalence for all $X\in Sp^G$.
\end{definition}
Recall that Bousfield localization at $E$ determines for each $X\in Sp^G$ a cofiber sequence
\[Z_E(X)\x{\psi_X} X\x{\phi_X} L_E(X)\]
with $Z_E(X)\in\mc Z_E$ and $L_E(X)\in \mc L_E$, which is unique up to homotopy with respect to these properties.
\begin{proposition}\label{2.7}
The following characterizations of smashing localizations are equivalent:
\begin{enumerate}
\item $L_E$ is smashing.
\item $\mc L_E$ is closed under homotopy colimits.
\item $\mc L_E$ is closed under arbitrary coproducts.
\item $\mc L_E$ is a smash ideal. That is $X\in\mc L_E$, $Y\in Sp^G\implies X\smsh Y\in\mc L_E$.
\item If $R\in\mc L_E$ is a ring spectrum, every $R$-module is in $\mc L_E$.
\item $\ip{E}=\ip{L_E(S^0)}$
\end{enumerate}
\end{proposition}
\begin{proof}
For $1\iff 2\iff 3$ see \cite{Lurie}. We show $1\implies 4\implies 5\implies 6\implies 1$: If $L_E$ is smashing, then if $X\in \mc L_E$ and $Y\in Sp^G$, 
\[X\smsh Y\simeq L_E(X)\smsh Y\simeq L_E(S^0)\smsh X\smsh Y\simeq L_E(X\smsh Y)\in \mc L_E\]
If $\mc L_E$ is a smash ideal, $R\in\mc L_E$ is a ring spectrum, and $M$ is an $R$-module, then $M$ is a retract of $R\smsh M$, which must be local, and $\mc L_E$ is closed under retracts. Note that $L_E$ is lax monoidal (on the level of the homotopy category), hence $L_E(S^0)$ is a ring spectrum in $\mc L_E$. $\mc Z_{L_E(S^0)}\subset\mc Z_E$ is clear, and assuming (5), $Z\in\mc Z_E$ implies that $Z\smsh L_E(S^0)\in\mc Z_E$, and as a module over $L_E(S^0)$, $Z\smsh L_E(S^0)\in\mc L_E$, hence $Z\smsh L_E(S^0)\simeq*$, i.e. $Z\in\mc Z_{L_E(S^0)}$. Now, since for any $X\in Sp^G$, $X\to L_E(S^0)\smsh X$ becomes an equivalence after smashing with $E$, to show $L_E$ is smashing, it suffices to show $L_E(S^0)\smsh X\in\mc L_E$. But since $L_E(S^0)$ is a ring spectrum, $L_E(S^0)\smsh X\in\mc L_{L_E(S^0)}$ by \hyperref[2.5]{2.5}, but $\mc L_{L_E(S^0)}=\mc L_{E}$, assuming (6).
\end{proof}

We will prefer characterization (6), as it is the only one that is phrased as a condition on the category of $E$-acyclics, rather than $E$-locals. Smashing localizations were studied in a more general setting by Balmer and Favi in \cite{Balmer}, and we recall here some of their definitions and results.
\begin{definition}(\cite{Balmer}, Definition 3.2)\label{2.8}
Let $(\mc T,\otimes,\mathds{1})$ be a tensor-triangulated (tt-) category (e.g. $Ho(Sp^G)$). We say that a distinguished triangle in $\mc T$ of the form
\[e\x{\psi}\mathds{1}\x{\phi} f\to\Sigma e\]
is an idempotent triangle if it satisfies any of the following equivalent conditions:
\begin{enumerate}
\item $e\otimes f=0$
\item $(1_e\otimes\psi):e\otimes e\to e$ is an isomorphism. (Left Idempotent)
\item $(1_f\otimes\phi):f\to f\otimes f$ is an isomorphism. (Right Idempotent)
\end{enumerate}
\end{definition}
The relationship between idempotent triangles and smashing localizations is as follows.
\begin{definition}\cite{Balmer}\label{2.9}
Let $\mc T$ be a tt-category and $\mc J\subset \mc T$ a thick tensor ideal. We define
\[J^\perp=\{t\in\mc T\st \Hom_{\mc T}(z,t)=0\tx{ for all }z\in\mc J\}\]
We say that $\mc J$ is a Bousfield ideal if for every $t\in\mc T$, there exists a distinguished triangle
\[e_t\to t\to f_t\to\Sigma f_t\]
such that $e_t\in\mc J$ and $f_t\in\mc J^\perp$. We say that $\mc J$ is a smashing ideal if $\mc J^\perp$ is a tensor ideal.
\end{definition}
\begin{theorem}\label{2.10}
$($\cite{Balmer}\textnormal{, Theorem 3.5}$)$
If $(\mc T,\otimes,\mathds{1})$ is a rigidly-compactly generated tt-category, there is a 1-1 correspondence between isomorphism classes of idempotent triangles and smashing ideals in $\mc T$, wherein $\mc J$ as above corresponds to the triangle
\[e_{\mathds{1}}\to\mathds{1}\to f_\mathds{1}\to\Sigma e_{\mathds{1}} \]
and an idempotent triangle
\[e\to \mathds{1}\to f\to\Sigma e\]
corresponds to the smashing ideal $\ker(-\otimes f)$.
\end{theorem}
\begin{corollary}\label{2.11}
If $(\mc T,\otimes,\mathds{1})=(Ho(Sp^G),\smsh,S^0)$, there is a 1-1 correspondence between isomorphism classes of idempotent triangles in $\mc T$ and smashing Bousfield classes $\ip{E}$, and hence also between smashing ideals in $Ho(Sp^G)$ and smashing Bousfield classes.
\end{corollary} 
\begin{proof}
Each smashing $\ip{E}$ determines the idempotent triangle
\[Z_E(S^0)\to S^0\to L_E(S^0)\to\Sigma Z_E(S^0)\]
as $L_E(S^0)\smsh Z_E(S^0)\simeq L_E(Z_E(S^0))\simeq*$. Conversely, if $e\to S^0\to f\to\Sigma e$ is an idempotent triangle, then it follows that it is isomorphic to
\[Z_f(S^0)\to S^0\to L_f(S^0)\to\Sigma Z_f(S^0)\]
and therefore corresponds to the smashing Bousfield class $\ip{f}$. Indeed, $f$ is a ring spectrum via the isomorphism $f\otimes f\cong f$, so $f$ is $f$-local, and the map $S^0\to f$ is therefore isomorphic as a right idempotent to the map $S^0\to L_f(S^0)$. These give mutually inverse maps of posets because if $\ip{E}$ is smashing, then $\ip{E}=\ip{L_E(S^0)}$ by \hyperref[2.7]{2.7}, and conversely we have just shown that $f\cong L_f(S^0)$.
\end{proof}
\begin{corollary}\label{2.12}
If $E_1,\ldots,E_n\in Sp^G$ are all smashing, then so are $E_1\smsh \cdots\smsh E_n$ and $E_1\vee\cdots\vee E_n$. Moreover, $Z_{E_1\vee\cdots\vee E_n}(S^0)\simeq Z_{E_1}(S^0)\smsh\cdots\smsh Z_{E_n}(S^0)$ and $L_{E_1\smsh\cdots\smsh E_n}(S^0)\simeq L_{E_1}(S^0)\smsh\cdots\smsh L_{E_n}(S^0)$.
\end{corollary}
\begin{proof}
It is shown in \cite{Balmer} that the tensor product gives the product in the category of left idempotents and the coproduct in the category of left idempotents. It follows from \hyperref[2.10]{2.10} then that the poset of smashing ideals in $Sp^G$ has meets and joins, and if $E,F$ are smashing $G$-spectra, these correspond to $E\vee F$ and $E\smsh F$ respectively. 
\end{proof}

\section{Bousfield Localizations and Change of Group}\label{sec3}
In this section, we start with a $G$-spectrum $E$ and explore the Bousfield localization functors associated to the spectrum $F(E)$ along various change of group functors $F$. We explore whether $F(E)$ is smashing, assuming that $E$ is smashing.  
\subsection{The General Case.}\label{sec3.1} We first establish some elementary facts about the behavior of localization functors along change of group functors $F$ in general.
\begin{definition}\label{3.1}
Let $i_*:Sp\to Sp^G$ denote the functor that sends a spectrum to the corresponding $G$-spectrum with trivial action, and $(-)^G:Sp^G\to Sp$ its right adjoint, the genuine fixed points. For a subgroup $H\subset G$, let $i^G_H:Sp^G\to Sp^H$ and $G_+\smsh_H(-):Sp^H\to Sp^G$ denote the restriction and induction functors respectively. 
\end{definition}
\begin{proposition}\label{3.2}
We have the following description of localization functors:
\begin{enumerate}
\item $L_{i_*E}(i_*X)\simeq i_*L_E(X)$ for any $E,X\in Sp$
\item $L_{i^G_HE}(i^G_HX)\simeq i^G_HL_E(X)$ for any $E,X\in Sp^G$
\end{enumerate}
\end{proposition}
\begin{proof}
$i_*$ and $i^G_H$ are symmetric monoidal, hence the map $i_*X\to i_*L_E(X)$ becomes an equivalence after smashing with $i_*E$, and $i^G_HX\to i^G_HL_EX$ becomes an equivalence after smashing with $i^G_HE$. $i_*L_E(X)$ is $i_*E$-local because if $Z\in\mc Z_{i_*E}$, then 
\[[Z,i_*L_E(X)]^G\cong[Z_G,L_E(X)]\]
and $Z_G\in \mc Z_E$ because
\[Z_G\smsh E\simeq (Z\smsh i_*E)_G\simeq *\]
$i^G_HL_E(X)$ is $i^G_HE$-local because if $Z\in \mc Z_{i^G_HE}$, then
\[[Z,i^G_HL_E(X)]^H\cong[G_+\smsh_HZ,L_E(X)]^G\]
and $G_+\smsh_HZ\in \mc Z_E$, as
\[(G_+\smsh_HZ)\smsh E\simeq G_+\smsh_H(Z\smsh i^G_HE)\simeq*\]
%
%
\end{proof}

From \hyperref[2.2]{2.2}, it is not difficult in general to characterize the $F(E)$-acyclics in terms of the $E$-acyclics, where $F$ is one of our change of group functors above. Characterizing the $F(E)$-locals in terms of the $E$-locals is much more difficult. For restriction and induction, however, we can give a simple necessary and sufficient condition.

\begin{proposition}\label{3.3}
For any $E\in Sp^G$, $Y\in Sp^H$ is $i^G_HE$-local if and only if $G_+\smsh_HY$ is $E$-local.
\end{proposition}
\begin{proof}
If $Y$ is $i^G_HE$-local, then if $Z\in\mc Z_E$, we have
\[[Z,G_+\smsh_HY]^G\cong[i^G_HZ,Y]^H=0\]
Conversely, if $Z\in \mc Z_{i^G_HE}$, $G_+\smsh_H\mc Z_{i^G_HE}\subset \mc Z_{E}$ implies that
\[[Z,i^G_H(G_+\smsh_HY)]^H\cong [G_+\smsh_HZ,G_+\smsh_HY]^G=0\]
and since $Y$ is a summand of $i^G_H(G_+\smsh_HY)$, $[Z,Y]^H=0$ .
\end{proof}
\begin{definition}\label{3.4}
Let $H\subset G$, then we let $\mc F_H$ be the family of subgroups of $G$ that are subconjugate to $H$ - that is, $\mc F_H$ is the smallest family of subgroups of $G$ containing $H$. We say a $G$-spectrum $X$ is $H$-cofree if the canonical map
\[X\to F(E{\mc {F}_H}_+,X)\]
is an equivalence, where $F(-,-)$ denotes the internal mapping spectrum in $Sp^G$, and $E{\mc {F}_H}$ is the universal $G$-space for the family $\mc F_H$. We simply say cofree, or Borel complete, when $H=\{e\}$. 
\end{definition}
\begin{lemma}\label{3.5}
If $X\in Sp^G$, then $F(E{\mc {F}_H}_+,X)\simeq L_{G/H_+}(X)$.
\end{lemma}
\begin{proof}
Since $i^G_H(E{\mc {F}_H}_+)\simeq S^0$, it follows that
\[X\to F(E{\mc {F}_H}_+,X)\]
becomes an equivalence after smashing with $G/H_+$. $F(E{\mc {F}_H}_+,X)$ is $G/H_+$-local because if $Z\in \mc Z_{G/H_+}$ so that $i^G_HZ\simeq *$, then
\[[Z,F(E{\mc {F}_H}_+,X)]^G\cong[Z\smsh E{\mc {F}_H}_+,X]^G\]
and $Z\smsh E{\mc {F}_H}_+\simeq *$. For this, let
\[\mc T=\{Y\in Sp^G\st Z\smsh Y\simeq *\}\]
then $\mc T$ is a localizing subcategory of $Sp^G$, and $E{\mc {F}_H}_+$ is in the localizing subcategory generated by $\{G/K_+\st K\in\mc F_H\}$, so it suffices to observe that $G/K_+\in\mc T$ for all $K\in \mc{F}_H$. 
\end{proof}
\begin{corollary}\label{3.6}
A map $f:X\to Y$ in $Sp^G$ between $H$-cofree $G$-spectra is an equivalence if and only if $i^G_H(f)$ is an equivalence.
\end{corollary}
\begin{proof}
In general, a map between $E$-locals is an equivalence if and only if it becomes an equivalence after smashing with $E$. Letting $E=G/H_+$ gives the result.
\end{proof}
\begin{proposition}\label{3.7}
For any $E\in Sp^H$, $X\in Sp^G$ is $G_+\smsh_HE$-local if and only if $X$ is $H$-cofree and $i^G_HX$ is $i^G_H(G_+\smsh_HE)$-local. \end{proposition}
\begin{proof}
Suppose $X$ is $G_+\smsh_HE$-local. Clearly, $\mc Z_{G/H_+}\subset\mc Z_{G_+\smsh_H E}$ and hence $\mc L_{G_+\smsh_HE}\subset\mc L_{G/H_+}$ - that is, $G_+\smsh_HE$-locals are $H$-cofree. $i^G_HX$ is $i^G_H(G_+\smsh_HE)$-local by \hyperref[3.2]{3.2}.

Conversely, if $X$ is $H$-cofree, then it suffices to show the map $\phi_X:X\to L_{G_+\smsh_HE}(X)$ is an equivalence, and by \hyperref[3.6]{3.6}, it suffices to show that $i^G_H(\phi_X)$ is an equivalence, which follows again by assumption from \hyperref[3.2]{3.2}.
\end{proof}
\begin{remark}\label{3.8}
Since $E$ is a retract of $i^G_H(G_+\smsh_HE)$, we have $\mc Z_{i^G_H(G_+\smsh_HE)}\subset\mc Z_{E}$ and hence $\mc L_{E}\subset \mc L_{i^G_H(G_+\smsh_HE)}$. For $X$ to be $G_+\smsh_HE$-local, it is therefore \ti{sufficient} for $X$ to be $H$-cofree and $i^G_HX$ to be $E$-local.
\end{remark}
The following are easy consequences of the double coset formula for $i^G_H(G_+\smsh_HE)$.
\begin{corollary}\label{3.9}
If $H\subset G$ is normal, $X\in Sp^G$ is $G_+\smsh_HE$-local if and only if $X$ is $H$-cofree and $i^G_HX$ is $\bigvee\limits_{[g]\in G/H} {^gE}$-local, where $^gE$ are the Weyl conjugates of $E$.
\end{corollary}
\begin{corollary}\label{3.10}
If $G$ is abelian, $X\in Sp^G$ is $G_+\smsh_HE$-local if and only if $X$ is $H$-cofree and $i^G_HX$ is $E$-local.
\end{corollary}
\subsection{The Smashing Case.}\label{sec3.2} We now discuss how smashing localizations behave under change of group functors. We first recall the following variant of the norm functor $N_H^G:Sp^H\to Sp^G$ of \cite{HHR}. Let $N^{G/H}:Sp^G\to Sp^G$ denote the composition $N_H^G\circ i^G_H$, and for
\[T=G/H_1\sqcup\cdots\sqcup G/H_n\]
a finite $G$-set, we let $N^T:Sp^G\to Sp^G$ denote the functor $N^{G/H_1}\smsh\cdots\smsh N^{G/H_n}$. We will also need the following description of how geometric fixed points interact with the norm.
\begin{proposition}\label{3.11}
$($\cite{HHR}\textnormal{, Proposition B.209}$)$ For any $K,H\subset G$, and for any $E\in Sp^H$, the diagonal gives an equivalence of spectra
\[\Phi^KN_H^GE\x{\simeq}\bigwedge\limits_{[g]\in K\backslash G/H}\Phi^{K^g\cap H}E\]
\end{proposition}
\begin{proposition}\label{3.12} Let $H\subset G$ be a subgroup. Smashing Bousfield classes are preserved by the following change of group functors:
\begin{enumerate}
\item If $E\in Sp$ is smashing, then $i_*E\in Sp^G$ is smashing.
\item If $E\in Sp^G$ is smashing, then $i^G_HE\in Sp^H$ is smashing.
\item If $E\in Sp^G$ is smashing, then $\Phi^H(E)\in Sp$ is smashing.
\item Let $f:G\to G'$ be a group homomorphism and $f^*:Sp^{G'}\to Sp^G$ the induced functor. If $E\in Sp^{G'}$ is smashing, then $f^*E\in Sp^G$ is smashing.
\item If $E\in Sp^H$ is smashing, then $N_H^GE$ is smashing.
\item If $E\in Sp^G$ is smashing, and $T$ is a finite $G$-set, then $N^TE$ is smashing.
\item If $E\in Sp^G$ is smashing, and for all $H\subset G$, $\mc Z_{E^G}\subset\mc Z_{E^H}$ (e.g. if $E$ is a ring spectrum), then $E^G$ is smashing.
\end{enumerate}
Moreover, for each functor $F$ in items (1)-(6), we have $L_{F(E)}(F(X))\simeq F(L_E(X))$. In item (7), we have
\[Z_{E^G}(X)=\bigotimes\limits_{H\subset G}\Phi^H(Z_R(X))\]
so that
\[L_{E^G}(X)=\mathrm{cofib}\bg(\bigotimes\limits_{H\subset G}\Phi^H(Z_R(X))\to S^0\bg)\]
\end{proposition}
\begin{proof}
In all cases, we have a smashing spectrum $E$ and therefore $\ip{E}=\ip{L_E(S^0)}$. If $F$ is one of the functors listed in items (1)-(6), it is symmetric monoidal, and hence $F(L_E(S^0))$ is a right idempotent, so it suffices to show $\ip{F(L_E(S^0))}=\ip{F(E)}$.
For (1), the relation $\Phi^H\circ i_*\simeq\id_{Sp}$ for all $H\subset G$ gives
\[\ip{\Phi^H(i_*L_E(S^0))}=\ip{L_E(S^0)}=\ip{E}=\ip{\Phi^H(i_*E)}\implies \ip{i_*L_E(S^0)}=\ip{i_*E}\]
For (2), note that
\[Z\in \mc Z_{i^G_HL_E(S^0)}\iff(G_+\smsh_HZ)\smsh L_E(S^0)\simeq *\iff (G_+\smsh_HZ)\smsh E\simeq *\]
since $E$ is smashing. $i^G_HE$ is shown to have the same acyclics by an identical argument. For (3), we have
\aln{
Z\smsh \Phi^G(E)\simeq*&\iff \tilde{E}\mc P\smsh i_*Z\smsh E\simeq*\\
&\iff \tilde{E}\mc P\smsh i_*Z\smsh L_E(S^0)\simeq*\\
&\iff Z\smsh\Phi^G(L_E(S^0))\simeq *
}
For (4), if $H$ is any subgroup of $G'$, the relation $\Phi^H\circ f^*=\Phi^{f(H)}$ gives
 \[\ip{\Phi^H(f^*L_E(S^0))}=\ip{\Phi^{f(H)}L_E(S^0)}=\ip{L_{\Phi^{f(H)}(E)}(S^0)}=\ip{\Phi^{f(H)}(E)}=\ip{\Phi^H(f^*E)}\]
by applying case (3). For (5), \hyperref[3.11]{3.11} gives 
\aln{\ip{\Phi^K(N_H^G(L_E(S^0)))}&=\ip{\bigwedge\limits_{[g]\in K\backslash G/H}\Phi^{K^g\cap H}(L_E(S^0))}\\
&=\bigwedge\limits_{[g]\in K\backslash G/H}\ip{\Phi^{K^g\cap H}(L_E(S^0))}\\
&=\bigwedge\limits_{[g]\in K\backslash G/H}\ip{\Phi^{K^g\cap H}(E)}\\
&=\ip{\Phi^K(N_H^G(E))}
}
for any subgroup $K\subset G$, again by applying case (3). For (6), if $T=G/H$, the result follows by combining cases (2) and (5), and the general case follows from \hyperref[2.12]{2.12}. The final remark follows in these cases from the localizations being smashing, as then the condition may be checked on the sphere spectrum.

For the genuine fixed points functor (7), we have $\mc Z_{E^G}\subset Z_{E^H}$ for all $H\subset G$, by assumption, hence 
\[Z\in\mc Z_{E^G}\iff i_*Z\in \mc Z_{E}\]
and this holds if and only if $Z\smsh\Phi^H(E)\simeq*$ for all $H\subset G$. We find:
\[\ip{E^G}=\ip{\bigvee\limits_{H\subset G}\Phi^H(E)}\]
and the claim follows as in \hyperref[2.12]{2.12}. The assumptions hold for $E$ a ring spectrum because one has restriction ring maps $E^G\to E^H$ given by applying $(-)^G$ to the map of rings
\[E\to F({G/H}_+,S^0)\smsh E\]
\end{proof}
\begin{remark}\label{3.13}
We needed to assume $E$ is smashing in \hyperref[3.12]{3.12} to establish that $\Phi^G(L_E(S^0))$ is $\Phi^G(E)$-local, whereas with $i^G_H$ and $i_*$, we could exploit the existence of a left adjoint to get around this assumption. In fact, it is not necessarily true that $\Phi^G(L_E(S^0))$ is $\Phi^G(E)$-local without this assumption. For example, if $G=C_2$, $E={C_2}_+$, $X=S^0$, then the left hand side is a point, and the right hand is $(S^0)^{tC_2}$. This example also shows us that the converse to case (3) of \hyperref[3.12]{3.12} is false, i.e. we cannot detect whether $E$ is smashing just by knowing that $\Phi^HE$ is smashing for all $H\subset G$.
\end{remark}
\begin{corollary}\label{3.14}
We have the following characterizations of local objects for smashing localizations:
\begin{enumerate}
\item If $E\in Sp^G$ is smashing, $X\in Sp^G$ is $E$-local if and only if $\Phi^H(X)$ is $\Phi^H(E)$-local for all $H\subset G$.
\item If $E\in Sp$ is smashing, $X\in Sp^G$ is $i_*E$-local if and only if $\Phi^H(X)$ is $E$-local for all $H\subset G$.
\item If $f:G\to G'$ is a group homomorphism, and $E\in Sp^{G'}$ is smashing, $X\in Sp^G$ is $f^*E$-local if and only if $\Phi^H(X)$ is $\Phi^{f(H)}(E)$-local for all $H\subset G$.
\item If $H\subset G$, and $E\in Sp^H$ is smashing, $X\in Sp^G$ is $N_H^GE$-local if and only if for all $K\subset G$, and for all $[g]\in K\backslash G/H$, $\Phi^K(X)$ is $\Phi^{K^g\cap H}(E)$-local.
\item If $E\in Sp^G$ is smashing, and $\mc Z_{E^G}\subset\mc Z_{E^H}$ for all $H\subset G$ (e.g. if $E$ is a ring spectrum), then $X\in Sp^G$ is $E^G$-local if and only if $i_*X$ is $E$-local.
\end{enumerate}
\end{corollary}
\begin{proof}
For (1), $X$ is $E$-local iff the map $X\to L_E(X)$ is an equivalence, but this is true iff 
\[\Phi^H(X)\to \Phi^H(L_E(X))\simeq L_{\Phi^H(E)}(X)\]
is an equivalence for all $H$, i.e. $\Phi^H(X)$ is $\Phi^H(E)$-local for all $H$. The rest are immediate consequences of (1).
\end{proof}

The norm is unique among the above functors in that it does not in general preserve cofiber sequences. However, we have the following interesting corollary of \hyperref[3.12]{3.12}:
\begin{corollary}\label{3.15}
If $G$ is abelian, $N_H^G$ preserves idempotent cofiber sequences. That is, if $e\to S^0\to f\to\Sigma e$ is an idempotent triangle in $Sp^H$, then $N_H^G(e)\to S^0\to N_H^G(f)$ is a cofiber sequence in $Sp^G$ such that
\[N_H^G(e)\to S^0\to N_H^G(f)\to\Sigma N_H^G(e)\]
is an idempotent triangle.
\end{corollary}
\begin{proof}
By \hyperref[2.11]{2.11}, every idempotent triangle in $Sp^H$ is of the form
\[Z_E(S^0)\to S^0\to L_E(S^0)\to\Sigma Z_E(S^0)\]
We will show that the sequence
\[N_H^G(Z_E(S^0))\to S^0\to N_H^G(L_E(S^0))\]
is equivalent to the idempotent cofiber sequence
\[Z_{N_H^GE}(S^0)\to S^0\to L_{N_H^GE}(S^0)\]
Note that since $N_H^G(*)=*$, $N_H^G(-)$ sends the zero map to the zero map. Therefore the composite $N_H^G(Z_E(S^0))\to S^0\to N_H^G(L_E(S^0))$ is null, and so we have a commutative diagram
\[
\btz
N_H^G(Z_E(S^0))\arrow[d,dashed,"f"]\arrow[r]&S^0\arrow[d,"="]\arrow[r]&N_H^G(L_E(S^0))\arrow[d,"\simeq"]\\
Z_{N_H^GE}(S^0)\arrow[r]& S^0\arrow[r]&L_{N_H^GE}(S^0)
\etz
\]
To show that $f$ is an equivalence, it suffices to show that $\Phi^K(f)$ is an equivalence for all $K\subset G$, hence it suffices to show $\Phi^K(-)$ of the top row is a cofiber sequence. This gives
\[\bigwedge\limits_{[g]\in K\backslash G/H}\Phi^{K\cap H}(Z_E(S^0))\to S^0\to \bigwedge\limits_{[g]\in K\backslash G/H}\Phi^{K\cap H}(L_E(S^0))\]
by \hyperref[3.11]{3.11}, where we have used that $G$ is abelian so that $K^g=K$. By \hyperref[3.12]{3.12}, this may be further identified with
\[\bigwedge\limits_{[g]\in K\backslash G/H}Z_{\Phi^{K\cap H}(E)}(S^0)\to S^0\to \bigwedge\limits_{[g]\in K\backslash G/H}L_{\Phi^{K\cap H}(E)}(S^0)\]
By \hyperref[2.12]{2.12}, this is the idempotent cofiber sequence associated to $\mc Z_{\Phi^{K\cap H}(E)}$, as
\[\ip{\Phi^{K\cap H}(E)}=\ip{\bigvee\limits_{[g]\in K\backslash G/H}\Phi^{K\cap H}(E)}=\ip{\bigwedge\limits_{[g]\in K\backslash G/H}\Phi^{K\cap H}(E)}\]
since $\Phi^{K\cap H}(E)$ is smashing.
\end{proof}

\begin{remark}\label{3.16}
It doesn't make sense to ask whether $N_H^G$ preserves idempotent \ti{triangles} in the sense of \cite{Balmer} because $N_H^G(S^1)\simeq S^{\tx{Ind}_H^G(1)}$, and so applying $N_H^G$ to the idempotent triangle
\[e\to S^0\to f\to\Sigma e\]
yields the sequence of maps
\[N_H^G(e)\to S^0\to N_H^G(f)\to S^{\tx{Ind}_H^G(1)}\smsh N_H^G(e)\]
which is not a distinguished triangle in $Sp^G$ unless $H=G$ or $e\simeq*$. We have only shown that, when $G$ is abelian,
\[N_H^G(e)\to S^0\to N_H^G(f)\]
is a cofiber sequence, and in particular the first two morphisms in an idempotent triangle. 
\end{remark}

We now give a counterexample to the above claim in the general case when $G$ is not necessarily abelian.
\begin{proposition}\label{3.17}
Fix an inclusion $C_2\inj\Sigma_3$. The corresponding functor $N_{C_2}^{\Sigma_3}:Sp^{C_2}\to Sp^{\Sigma_3}$ does not preserve all idempotent cofiber sequences.
\end{proposition}
\begin{proof}
Consider the idempotent cofiber sequence
\[E{C_2}_+\to S^0\to \tilde{E}C_2\]
in $Sp^{C_2}$. Applying $N_{C_2}^{\Sigma_3}$ yields the sequence
\[E{\mc F_{C_3}}_+\to S^0\to \tilde{E}\mc P\]
which is not a cofiber sequence.
\end{proof}

\subsection{Induction and Smashing}\label{sec3.3} By far the most interesting change of group functor with respect to smashing localizations is induction, because it is not monoidal, and hence we treat it separately. We find that induced $G$-spectra $G_+\smsh_HE$ are rarely smashing, though we give a necessary and sufficient condition for $G_+\smsh_HE$ to be smashing.
\begin{proposition}\label{3.18}
Suppose $H\subset G$, and $E\in Sp^H$. Then $G_+\smsh_HE$ is smashing if and only if $\Phi^K(L_{G_+\smsh_HE}(S^0))\simeq*$ for all $K\notin\mc F_H$ and $i^G_H(G_+\smsh_HE)$ is smashing.
\end{proposition}
\begin{proof}
Suppose $G_+\smsh_HE$ is smashing, then any restriction of it is also smashing. If $K\notin\mc F_H$,  let $\mc F$ be the smallest family containing $H$ and every proper subgroup of $K$. It follows that $\tilde{E}\mc F\in\mc Z_{G/H_+}\subset \mc Z_{G_+\smsh_HE}$, as $i^G_H\tilde{E}\mc F\simeq*$. Since $G_+\smsh_HE$ is smashing, $\mc Z_{G_+\smsh_HE}=\mc Z_{L_{G_+\smsh_HE}(S^0)}$, so that
\[\tilde{E}\mc F\smsh L_{G_+\smsh_HE}(S^0)\simeq*\implies i^G_K\tilde{E}\mc F\smsh i^G_KL_{G_+\smsh_HE}(S^0)\simeq*\]
Since $\mc F$ contains every proper subgroup of $K$, but not $K$ itself (as $K\notin \mc F_H$), we have $i^G_K\tilde{E}\mc F\simeq\tilde{E}\mc P$, where the latter is the universal $K$-space for the family of proper subgroups of $K$. This by definition implies $\Phi^K(L_{G_+\smsh_HE}(S^0))\simeq*$.

Conversely, suppose $\Phi^K(L_{G_+\smsh_HE}(S^0))\simeq*$ for all $K\notin\mc F_H$ and $i^G_H(G_+\smsh_HE)$ is smashing. To show $L_{G_+\smsh_HE}$ is smashing, it suffices to show that $\mc L_{G_+\smsh_HE}$ is closed under arbitrary coproducts. We note first that if $Y\in \mc L_{G_+\smsh_HE}$, $Y$ is a module over $L_{G_+\smsh_HE}(S^0)$, hence $\Phi^K(Y)\simeq*$ for all $K\notin\mc F_H$. Let $\{Y_i\}$ be a set of $G_+\smsh_HE$-locals. Consider the map
\[\phi:\bigvee\limits_i Y_i\to L_{G_+\smsh_HE}\bg(\bigvee\limits_iY_i\bg)\]
It suffices to show this map is an equivalence, and it becomes an equivalence after applying $\Phi^K$ for all $K\notin\mc F_H$ since it is the identity map of a point, up to equivalence. If $K\in\mc F_H$, then $gKg^{-1}=H'$ for some $g\in G$ and some $H'\subset H$. Since $\Phi^K(-)\simeq\Phi^{H'}(-)$, it suffices to assume $K\subset H$, in which case $\Phi^K$ factors through the functor $i^G_H$, and $i^G_H(\phi)$ is the map
\[\bigvee\limits_ii^G_HY_i\to L_{i^G_HG_+\smsh_HE}\bg(\bigvee\limits_ii^G_HY_i\bg)\]
This is an equivalence as $i^G_H(G_+\smsh_HE)$ is assumed smashing and $i^G_HY_i$ is $i^G_H(G_+\smsh_HE)$-local.
\end{proof}
\begin{corollary}\label{3.19}
If $H\subset G$ is normal, and $E\in Sp^H$ is smashing, then $G_+\smsh_HE$ is smashing if and only if $\Phi^K(L_{G_+\smsh_HE}(S^0))\simeq*$ for all $K\subset H$.
\end{corollary}
\begin{proof}
This follows immediately from the previous proposition along with the observation that
\[\ip{i^G_H(G_+\smsh_HE)}=\ip{\bigvee\limits_{[g]\in G/H}{^gE}}=\bigvee\limits_{[g]\in G/H}\ip{^gE}\]
is a smashing Bousfield class by \hyperref[2.12]{2.12} since $^gE$ is smashing for all $g$.
\end{proof}
When $H\subset G$ is normal, we arrive at a somewhat explicit formula for an induced localization, which we can interpret as follows: induced smashing localizations are smashing \ti{after} $H$-cofree completion. When $H=\{e\}$, this can be further related to the corresponding trivial localization.
\begin{proposition}\label{3.20}
If $H\subset G$ is normal, $E\in Sp^H$ is smashing, and $X\in Sp^G$, then 
\[L_{G_+\smsh_HE}(X)\simeq L_{{G/H}_+}(L_{G_+\smsh_HE}(S^0)\smsh X)\simeq F(E{\mc{F}_H}_+,L_{G_+\smsh_HE}(S^0)\smsh X)\]
\end{proposition}
\begin{proof}
The map
\[X\to F(E{\mc{F}_H}_+,L_{G_+\smsh_HE}(S^0)\smsh X)\]
is a $G_+\smsh_HE$ equivalence since $i^G_H(G_+\smsh_HE)$ is smashing, and the target is easily seen to be $G_+\smsh_HE$-local from \hyperref[3.7]{3.7}.
\end{proof}

\begin{proposition}\label{3.21}
Let $E\in Sp$ be any spectrum, and $X\in Sp^G$, then 
\[L_{G_+\smsh E}(X)\simeq F(EG_+,L_{i_*E}X)\simeq F(EG_+,i_*L_E(S^0)\smsh X)\]
\end{proposition}
\begin{proof}
The map $X\to F(EG_+,L_{i_*E}X)$ becomes an equivalence after smashing with ${G}_+\smsh E$, and the target is $G_+\smsh E$-local by \hyperref[3.7]{3.7}.
\end{proof}
\begin{corollary}\label{3.22}
Let $E\in Sp$ be a smashing spectrum, then $G_+\smsh E$ is smashing if and only if $(L_E(S^0))^{tH}\simeq *$ for all nontrivial subgroups $H\subset G$.
\end{corollary}
\begin{proof}
$G_+\smsh E$ is smashing if $\Phi^H(L_{G_+\smsh E}(S^0))\simeq *$ for all nontrivial subgroups $H$, but 
\aln{
\Phi^H(L_{G_+\smsh E}(S^0))&\simeq\Phi^H(F(EG_+,i_*L_E(S^0)))\\
&\simeq\Phi^H(F(EH_+,i_*L_E(S^0)))
}
which is a module over the ring $(L_E(S^0))^{tH}$. Conversely, if $G_+\smsh E$ is smashing, then $\mc Z_{G_+\smsh E}=\mc Z_{F(EG_+,i_*L_E(S^0))}$, but $\tilde{E}G\smsh G_+\smsh E\simeq*$, and
\[(L_E(S^0))^{tH}\simeq (\tilde{E}G\smsh F(EG_+,i_*L_E(S^0)))^H\]
\end{proof}
\begin{corollary}\label{3.23}
Let $E=E(n)$ at the prime $p$, then for all $G$ such that $p$ divides $|G|$, $G_+\smsh E$ is smashing if and only if $n=0$.
\end{corollary}
\begin{proof}
When $n=0$, $E(0)=H\Q=L_0(S^0)$, and $H\Q^{tH}\simeq*$ for all $H$ nontrivial. If $n>0$, then $G$ has an element of order $p$ and hence if $G_+\smsh E$ were smashing, we would necessarily have $(L_n(S^0))^{tC_p}\simeq*$. However, we know from \cite{Hovey} that this Tate spectrum is not contractible.
\end{proof}
We end this section with an example illustrating the necessity of the normality conditions in \hyperref[3.19]{3.19} and \hyperref[3.20]{3.20}. It shows that if $E\in Sp^H$, then $E$ and $i^G_H(G_+\smsh_HE)$ are not always Bousfield equivalent, and $E$ being smashing does not always guarantee that $i^G_H(G_+\smsh_HE)$ is smashing.
\begin{proposition}\label{3.24}
Let $G=\Sigma_4$ and $H=D_8=\ip{(1234),(13)}\subset\Sigma_4$. Then $\tilde{E}\mc F_{\ip{(1234)}}$ is a smashing $D_8$-spectrum, but $i^{\Sigma_4}_{D_8}({\Sigma_4}_+\smsh_{D_8}\tilde{E}\mc F_{\ip{(1234)}})$ is not smashing.
\end{proposition}
\begin{proof}
We have 
\[D_8=\{e,(13)(24),(12)(34),(14)(23),(1234),(1432),(13),(24)\}\subset \Sigma_4\]
\[D_8\backslash\Sigma_4/D_8=\{D_8,D_8(12)D_8\}\]
\[^{(12)}D_8=\{e,(13)(24),(12)(34),(14)(23),(1342),(1243),(14),(23)\}\]
\[D_8\cap {^{(12)}D_8}=\{e,(13)(24),(12)(34),(14)(23)\}=V_4\]
Therefore we have
\[i^{\Sigma_4}_{D_8}{\Sigma_4}_+\smsh_{D_8}\tilde{E}\mc F_{\ip{(1234)}}=\tilde{E}\mc F_{\ip{(1234)}}\vee\bg({D_8}_+\smsh_{V_4}i^{{^{(12)}D_8}}_{V_4}\bg({^{(12)}\tilde{E}\mc F_{\ip{(1234)}}}\bg)\bg)\]
$^{(12)}\tilde{E}\mc F_{\ip{(1234)}}$ is the universal $^{(12)}D_8$ space $\tilde{E}\mc F_{\ip{(1342)}}$, and hence $i^{{^{(12)}D_8}}_{V_4}\bg({^{(12)}\tilde{E}\mc F_{\ip{(1234)}}}\bg)$ is the universal $V_4$-space $\tilde{E}\mc F_{\ip{(14)(23)}}$. Therefore we may write
\[i^{\Sigma_4}_{D_8}{\Sigma_4}_+\smsh_{D_8}\tilde{E}\mc F_{\ip{(1234)}}=\tilde{E}\mc F_{\ip{(1234)}}\vee\bg({D_8}_+\smsh_{V_4}\tilde{E}\mc F_{\ip{(14)(23)}}\bg)\]
We now assume for the sake of contradiction that this $D_8$-spectrum is smashing, hence we restrict to $\ip{(1234)}\cong C_4$ to get a smashing $C_4$-spectrum
\[i^{D_8}_{\ip{(1234)}}\bg(\tilde{E}\mc F_{\ip{(1234)}}\vee\bg({D_8}_+\smsh_{V_4}\tilde{E}\mc F_{\ip{(14)(23)}}\bg)\bg)\simeq i^{D_8}_{\ip{(1234)}}\bg({D_8}_+\smsh_{V_4}\tilde{E}\mc F_{\ip{(14)(23)}}\bg)\]
One checks that
\[\ip{(1234)}\backslash D_8/V_4=\{\ip{(1234)}eV_4\}\]
\[\ip{(1234)}\cap V_4=\ip{(13)(24)}\]
so that 
\aln{
i^{D_8}_{\ip{(1234)}}\bg({D_8}_+\smsh_{V_4}\tilde{E}\mc F_{\ip{(14)(23)}}\bg)&\simeq\ip{(1234)}_+\smsh_{V_4\cap \ip{(1234)}}i^{V_4}_{V_4\cap \ip{(1234)}}\tilde{E}\mc F_{\ip{(14)(23)}}\\
&\simeq\ip{(1234)}_+\smsh_{\ip{(13)(24)}}\tilde{E}\mc F_{\ip{(14)(23)}\cap\ip{(13)(24)}}\\
&\simeq {C_4}_+\smsh_{C_2}\tilde{E}C_2
}
Now $\tilde{E}C_2$ is a smashing $C_2$-spectrum, and so by \hyperref[3.19]{3.19}, ${C_4}_+\smsh_{C_2}\tilde{E}C_2$ is a smashing $C_4$-spectrum if and only if 
\[\Phi^{C_4}(L_{{C_4}_+\smsh_{C_2}\tilde{E}C_2}(S^0))\simeq *\]
and the proof of \hyperref[3.20]{3.20} shows that $L_{{C_4}_+\smsh_{C_2}\tilde{E}C_2}(S^0)=F(E{\mc P}_+,\tilde{E}C_4)$, but
\[\Phi^{C_4}(F(E{\mc P}_+,\tilde{E}C_4))\simeq (S^0)^{tC_2}\not\simeq *\]
\end{proof}

\section{Consequences for Chromatic Localizations}\label{sec4}
\subsection{Proofs of Main Theorems}\label{sec4.1}
In this section, we prove Theorems \hyperref[1.1]{1.1} and \hyperref[1.2]{1.2}, namely that the analogs of the smash product theorem and the chromatic convergence theorem for the $E_\R(n)$'s hold only after Borel completion. We remark on  analogs of the nilpotence and thick subcategory theorems. We also discuss recent $C_{2^n}$-equivariant analogs of the $E_\R(n)$'s constructed in \cite{BHSZ}, and we identify their Bousfield classes and deduce Theorem \hyperref[1.3]{1.3}.
\begin{theorem}\label{4.1}
If $n>0$, then $E_\R(n)$ is not smashing. Moreover, for $X\in Sp^{C_2}$,
\[L_{E_\R(n)}(X)\simeq F(E{C_{2}}_+,i_*L_{E(n)}(S^0)\smsh X)\]
\end{theorem}
\begin{proof}
By \hyperref[2.4]{2.4}, we have
\[\ip{E_\R(n)}=\ip{{C_{2}}_+\smsh E(n)}\]
and now the result follows from \hyperref[3.21]{3.21} and \hyperref[3.23]{3.23}. For the claim about $\Phi^{C_2}(E_\R(n))$, $E_\R(n)$ is a module over ${MU_\R}[\ov{v_n}^{-1}]$, and $\Phi^{C_2}({MU_\R}[\ov{v_n}^{-1}])\simeq*$ as $\Phi^{C_2}(\ov{v_n})=0$ (\cite{HHR}, 5.50).
\end{proof}

\begin{theorem}\label{4.2}
If $X$ is a $2$-local finite $C_{2}$-spectrum, we have a diagram
\[
\btz
X\arrow[r]\arrow[dr]&F(E{C_{2}}_+,X)\arrow[d,"\simeq"]\\
&\varprojlim_n L_{E_\R(n)}(X)
\etz
\]
\end{theorem}
\begin{proof}
The category of cofree $C_{2}$-spectra is closed under homotopy limits, hence there exists a unique up to homotopy vertical map making the above diagram commute. As a map between cofree $C_{2}$-spectra, it is an equivalence if and only if it induces an underlying equivalence. The underlying map is an equivalence by the nonequivariant chromatic convergence theorem (see \cite{Ravenel}).
\end{proof}

\begin{remark}\label{4.3}
The analogs of the nilpotence and thick subcategory theorems also fail in genuine $C_2$-spectra, and this is much easier to see. For the nilpotence theorem, the class $2-[C_2]\in\pi_0^{C_2}(S^0)\cong A(C_2)$ goes to $0$ in $\pi_0^{C_2}(MU_\R)=\Z$, but it is not nilpotent in $A(C_2)$. Passing to Borel $C_2$-spectra does not correct this: the endomorphism ring of the unit in Borel $C_2$-spectra is $A(C_2)^{\hat{}}_I$, by Lin's theorem \cite{Lin}, and $2-[C_2]$ is still not nilpotent.

In the case of the thick subcategory theorem, the Balmer spectrum of $(Sp^{C_2})^\omega$ was determined in \cite{BS}, and as remarked in the introduction, for $n>0$, none of the thick tensor ideals is the collection of finite acyclics of $E_\R(n)$, so no reasonable analog of the thick subcategory theorem for the $E_\R(n)$'s (or the $K_\R(n)$'s) can hold in $Sp^{C_2}$. Passing to Borel, we run into the following issue: the unit is not compact in Borel $C_2$-spectra, and in particular the compact objects and dualizable objects do not coincide. This makes an analysis of the spectrum more difficult, but is a subject we plan to revisit in future work. 
\end{remark}

In \cite{BHSZ}, Beaudry, Hill, Shi, and Zeng construct genuine $C_{2^n}$-spectra that serve as analogs to the $E_\R(n)$'s. We recall their construction, which hinges on the observation that one may construct the $C_{2^n}$-spectrum
\[BP^{((G))}\ip{m}:=N_{C_2}^{C_{2^n}}BP_\R/(C_{2^n}\cdot \ov{t_{m+1}},C_{2^n}\cdot \ov{t_{m+2}},\ldots)\]
and for a carefully chosen class $D\in\pi^{G}_{*\rho_{G}}BP^{((G))}$, $D^{-1}BP^{((G))}\ip{m}$ should have height $h=2^{n-1}m$.  More precisely, they show:
\begin{theorem}\label{4.4}
$($\cite{BHSZ}\textnormal{, Theorems 1.5 and 1.8}$)$ For $h=2^{n-1}m$, there is a class $D\in\pi_{*\rho_{G}}^{G}BP^{((G))}$ and a height $h$ formal group law $\G_h$ over $\F_2$ such that for any perfect field $k$ of characteristic 2, if we regard the corresponding Lubin-Tate theory $E(k,\G_h)$ as a cofree $C_{2^n}$-spectrum, there is a diagram in $Sp^{C_{2^n}}$
\[
\btz
BP^{((G))}\arrow[r]\arrow[d]&E(k,\G_h)\\
D^{-1}BP^{((G))}\arrow[ur,dashed]
\etz
\]
\end{theorem}
It follows that the above map factors further through 
\[E_G(m):=D^{-1}BP^{((G))}\ip{m}\]
which can be thought of as a $C_{2^n}$-equivariant height $h$ Johnson-Wilson theory. The corresponding localization functors on $Sp^{C_{2^n}}$ behave formally very similarly to those of the $E_\R(n)$'s, as we can identify their Bousfield classes in a similar way. We need the following results about the class $D$:

\begin{theorem}\label{4.5}
$($\cite{BHSZ}\textnormal{, Theorem 1.2}$)$ The element 
\[i^{G}_eD\in \pi_*^eBP^{((G))}\ip{m}\cong\Z_{(2)}[G\cdot\ov{t_1},\ldots,G\cdot\ov{t_m}]\]
satisfies the following properties:
\begin{itemize}
\item $v_h$ divides $D$,
\item $(2,v_1,\ldots,v_h)$ is a regular sequence in $D^{-1}\pi_*^eBP^{((G))}\ip{m}$,
\item $v_r\in I_r$ for $r>h$,
\item $D^{-1}\pi_*^e(BP^{((G))}\ip{m})/I_h\cong\F_2[(t_m^{G})^{\pm}]$ with $v_h=t_m^{(2^h-1)/(2^m-1)}$, and
\item the formal group law carried by $\pi_*^eBP^{((G))}\ip{m}$ has height exactly h over $D^{-1}\pi_*^e(BP^{((G))}\ip{m})/I_h$.
\end{itemize}
\end{theorem}

\begin{corollary}\label{4.6}
The underlying spectrum of $E_G(m)$ is Bousfield equivalent to $E(h)$, and the geometric fixed points of $E_G(m)$ at any nontrivial subgroup is contractible. In particular,
\[\ip{E_G(m)}=\ip{{C_{2^n}}_+\smsh E(h)}\]
\end{corollary}
\begin{proof}
For the claim about geometric fixed points, \cite{BHSZ} (proof of Theorem 1.8) show that the class $D$ is divisible by norms of certain classes from $\pi_{*\rho_{C_2}}^{C_2}BP^{((G))}$, all of which become null upon applying $\Phi^{C_2}$, and the result follows as in (\cite{HHR}, Section 10). 

%
%
%
For the underlying spectrum, the conditions in \hyperref[4.5]{4.5} are enough to guarantee that the map
\[\Spec(D^{-1}\pi_*^eBP^{((G))}\ip{m})\to\mc M_{FG}\]
factors through a faithfully flat cover of the open substack $\mc M_{FG}^{\le h}$, and any such Landweber theory is Bousfield equivalent to $E(h)$.
In more detail, items (1) and (2) in \hyperref[4.5]{4.5} guarantee that the spectrum $i^{C_{2^n}}_e(D^{-1}BP^{((G))}\ip{m})$ is Landweber exact, and by functoriality it maps to the landweber exact spectrum $E$ with coefficient ring 
\[E_*:=(D^{-1}\pi_*^eBP^{((G))}\ip{m})[u]/(u^{2^m-1}-t_m)\]
with $|u|=2$, which is 2-periodic. $(D^{-1}\pi_*^eBP^{((G))}\ip{m})[u]/(u^{2^{m}-1}-t_m)$ is a free module over $D^{-1}\pi_*^eBP^{((G))}\ip{m}$, so the inclusion is faithfully flat, and the two Landweber theories are Bousfield equivalent. In $E_*$, we may use $u$ to conjugate the formal group into degree 0, and now a spectrum $X$ is $E$-acyclic if and only if the corresponding quasicoherent sheaves on $\mc M_{FG}$ determined by $E_0(X)$ and $E_1(X)$ are zero. It now suffices to show that the map 
\[\Spec(E_0)\to\mc M_{FG}^{\le h}\]
is a flat cover. 

This map is flat by the Landweber exact functor theorem, so it suffices to show that it is essentially surjective, and by \hyperref[4.4]{4.4}, there is a factorization
\[
\btz
\Spec(E(k,\G_h)_0)\arrow[r]\arrow[dr,"p"']&\Spec(E_0)\arrow[d]\\
&\mc M_{FG}^{\le h}
\etz
\]
and $p$ is a faithfully flat cover, as $\Spec(E(k,\G_h)_0)$ is a Lubin-Tate universal space of height $h$.
\end{proof}
The Bousfield classes of the $E_G(m)$'s are therefore nested: we see that 
\[\mc Z_{E_G(m)}\subset\mc Z_{E_G(m-1)}\]
and hence for any $X\in Sp^{C_2}$, we may form a chromatic tower
\[X\to\cdots\to L_{E_G(m)}(X)\to L_{E_G(m-1)}(X)\to\cdots \to L_{E_G(0)}(X)\]
Our results for $C_2$ now follow in essentially the same way for the $E_G(m)$'s:
\begin{theorem}\label{4.7}
Let $E_G(m)$ denote the $C_{2^n}$-spectrum $D^{-1}BP^{((G))}\ip{m}$ constructed in \cite{BHSZ}, where $h=2^{n-1}m$.
\begin{itemize}
\item If $m>0$, then $E_G(m)$ is not smashing. Moreover, for $X\in Sp^{C_{2^n}}$,
\[L_{E_G(m)}(X)\simeq F(E{C_{2^n}}_+,i_*L_{E(h)}(S^0)\smsh X)\]
\item If $X$ is a 2-local finite $C_{2^n}$-spectrum, we have a diagram
\[
\btz
X\arrow[r]\arrow[dr]&F(E{C_{2^n}}_+,X)\arrow[d,"\simeq"]\\
&\varprojlim_m L_{E_G(m)}(X)
\etz
\]
\end{itemize}
\end{theorem}

\subsection{Smashing $C_{p^n}$-Spectra}\label{sec4.2} In light of \hyperref[4.1]{4.1}, a natural question from here is then if the $E_\R(n)$ are not smashing, can we construct equivariant spectra analogous to the $E(n)$ that \ti{are} smashing? More specifically, every thick tensor ideal in $Sp^\omega_{(p)}$ is the collection of finite acyclics of one of the $E(n)$'s, so we may ask if a similar statement is true for $(Sp^G)^\omega_{(p)}$, and we can give a construction when $G=C_{p^n}$. The following theorem was proven in the case $n=1$ by Balmer and Sanders \cite{BS}, and for $n>1$ by Barthel, Hausmann, Naumann, Nikolaus, Noel, and Stapleton \cite{Barthel}.

\begin{theorem}\label{4.8}\cite{BS}\cite{Barthel}
The thick tensor ideals in $(Sp_{(p)}^{C_{p^n}})^c$ are precisely the subcategories of the form
\[\{X\in (Sp_{(p)}^{C_{p^n}})^c: \Phi^{C_{p^i}}(X)\in\mc Z_{E(m_i)}\}\]
where $m_i\le m_{i+j}+1$ for all $0\le i\le n-1$ and $1\le j\le n-i$.
\end{theorem}
Therefore for $G=C_{p^n}$, the above question becomes: for a sequence of natural numbers $m_0,\ldots,m_n$ with $m_i\le m_{i+j}+1$ for all $0\le i\le n-1$ and $1\le j\le n-i$, can we build a \ti{smashing} $G$-spectrum $E(m_0,\ldots,m_n)$ so that $\Phi^{C_{p^i}}(E(m_0,\ldots,m_n))\simeq E(m_i)$? It is easy to build $E(m_0,\ldots,m_n)$ with the stated geometric fixed points since we may assume by induction that $E(m_1,\ldots,m_n)\in Sp^{C_{p^{n-1}}}$ exists, and then set
\[E(m_0,\ldots,m_n):=(E{C_{p^n}}_+\smsh i_*E(m_0))\vee(\tilde{E}C_{p^n}\smsh q^*E(m_1,\ldots,m_n))\]
where $q:C_{p^n}\to C_{p^{n-1}}$ is the quotient map. It is not obvious that this spectrum is smashing, but using the results of Section \hyperref[sec3]{3}, we can build a different representative of the same Bousfield class that is manifestly smashing.

We do not know if there is a way to construct such spectra that are not split as above. However, what follows would show that any such construction produces a smashing $G$-spectrum, since it would be Bousfield equivalent to the ones we construct. We begin with the case $n=1$. 

\begin{proposition}\label{4.9}
For every pair of natural numbers $m_0,m_1$ satisfying $m_0\le m_1+1$, there is a smashing $C_p$-spectrum $E(m_0,m_1)$ with the property that 
\[\ip{\Phi^{C_{p^i}}(E(m_0,m_1))}=\ip{E(m_i)}\]
for $i=0,1$.
\end{proposition}
\begin{proof}
Setting $E(m_0,m_1)=({C_p}_+\smsh E(m_0))\vee(\tilde{E}C_p\smsh i_*E(m_1))$, one checks easily the claim about Bousfield classes. It suffices to show that for any family $\{Y_i\}$ of $E(m_0,m_1)$-locals, the map
\[\phi:\bigvee\limits_iY_i\to L_{m_0,m_1}\bg(\bigvee\limits_iY_i\bg)\]
is an equivalence. It induces an underlying equivalence as $i^{C_p}_{e}\circ L_{m_0,m_1}\simeq L_{m_0}\circ i^{C_p}_{e}$, so it suffices to show $\Phi^{C_p}(\phi)$ is an equivalence. If we knew that $\Phi^{C_{p}}(L_{m_0,m_1}(S^0))$ were $E(m_1)$-local, then $\Phi^{C_{p}}(L_{m_0,m_1}(Y))$ would be $E(m_1)$-local for any $Y$, as a module over an  $E(m_1)$-local ring spectrum. But then $\Phi^{C_{p}}(\phi)$ would be an $E(m_1)$-equivalence between $E(m_1)$-locals.

From Section \hyperref[sec3]{3}, we have
\aln{
L_{{C_p}_+\smsh E(m_0)}(X)&\simeq F(E{C_p}_+,i_*L_{m_0}(S^0)\smsh X)\\
L_{\tilde{E}C_p\smsh i_*E(m_1)}(X)&\simeq \tilde{E}C_p\smsh i_*L_{m_1}(S^0)\smsh X
}
It follows that $L_{{C_p}_+\smsh E(m_0)}\circ L_{\tilde{E}C_p\smsh i_*E(m_1)}\simeq*$, and hence by a general argument (see \cite{Bauer}), there is a natural homotopy pullback square
\[
\btz
L_{m_0,m_1}(X)\arrow[r]\arrow[d]&F(E{C_p}_+,i_*L_{m_0}(S^0)\smsh X)\arrow[d]\\
\tilde{E}C_p\smsh i_*L_{m_1}(S^0)\smsh X\arrow[r]&\tilde{E}C_p\smsh i_*L_{m_1}(S^0)\smsh F(E{C_p}_+,i_*L_{m_0}(S^0)\smsh X)
\etz
\]
Setting $X=S^0$, and applying $\Phi^{C_p}(-)$, we have a homotopy pullback square
\[
\btz
\Phi^{C_p}(L_{m_0,m_1}(S^0))\arrow[r]\arrow[d]&L_{m_0}(S^0)^{tC_p}\arrow[d]\\
L_{m_1}(S^0)\arrow[r]&L_{m_1}(L_{m_0}(S^0)^{tC_p})
\etz
\]
and by the main result of \cite{Hovey}, the right hand map is an equivalence if $m_0\le m_1+1$.
\end{proof}
\begin{theorem}\label{4.10}
For every sequence of natural numbers $m_0,\ldots,m_n$ satisfying $m_i\le m_{i+j}+1$ for all $0\le i\le n-1$ and $1\le j\le n-i$, there is a smashing $C_{p^n}$-spectrum $E(m_0,\ldots,m_n)$ with the property that
\[\Phi^{C_{p^i}}(E(m_0,\ldots,m_n))\simeq E(m_i)\]
for all $0\le i\le n$.  \end{theorem}
\begin{proof}
We proceed by induction on $n$, and we may assume $n>1$ by the previous proposition. As stated above, it suffices to show there is a spectrum  $E(m_0,\ldots,m_n)$ with the property that
\[\ip{\Phi^{C_{p^i}}(E(m_0,\ldots,m_n))}=\ip{ E(m_i)}\]
for all $i$. There are 3 cases to check:

(i) $m_0=m_1$: By induction, we may assume there is a smashing $C_{p^{n-1}}$-spectrum $E(m_1,\ldots,m_n)$ with the stated properties. Let $q:C_{p^n}\to C_{p^{n-1}}$ be the usual quotient map. Then $E(m_0,\ldots,m_n):=q^*E(m_1,\ldots,m_n)$ is a smashing $C_{p^n}$-spectrum and
\[\Phi^{C_{p^i}}(q^*E(m_1,\ldots,m_n))=\begin{cases}\Phi^{C_{p^{i-1}}}(E(m_1,\ldots,m_n))&i>0\\\Phi^{\{e\}}(E(m_1,\ldots,m_n))&i=0\end{cases}\]

(ii) $m_0<m_1$. Here we set 
\[E(m_0,\ldots,m_n):=i_*E(m_0)\vee(\tilde{E}C_{p^n}\smsh q^*E(m_1,\ldots,m_n))\] 
This is a smashing $C_{p^n}$-spectrum as in \hyperref[2.12]{2.12}, and we have
\[\Phi^{C_{p^i}}(E(m_0,\ldots,m_n))=\begin{cases}E(m_0)\vee\Phi^{C_{p^{i-1}}}(E(m_1,\ldots,m_n))&i>0\\E(m_0)&i=0\end{cases}\]
Note however that $m_0\le m_i$ for all $i>0$ as $m_i\ge m_1-1$ for all $i>1$, hence $\ip{E(m_0)\vee E(m_i)}=\ip{E(m_i)}$.

(iii) $m_0=m_1+1$. Since we have assumed $n>1$, we can form the smashing $C_{p^n}$-spectrum
\[E(m_0,\ldots,m_n):=N_{C_p}^{C_{p^{n+1}}}E(m_0,m_1)\vee(\tilde{E}C_{p^n}\smsh q^*E(m_1,\ldots,m_n))\]
and we have
\[\Phi^{C_{p^i}}(E(m_0,\ldots,m_n))=\begin{cases}\Phi^{C_p}(E(m_0,m_1))^{\smsh k(i)}\vee \Phi^{C_{p^{i-1}}}(E(m_1,\ldots,m_n))&i>0\\E(m_0)^{\smsh p^{n-1}} &i=0\end{cases}\]
where $k(i)$ is some positive integer that won't affect the Bousfield class. Note that since $m_0=m_1+1$, $m_i\ge m_1$ for all $i>0$, hence we have
\[\ip{\Phi^{C_{p^i}}(E(m_0,\ldots,m_n))}=\ip{E(m_1)}\vee\ip{E(m_i)}=\ip{E(m_i)}\]
for $i>0$, and
\[\ip{\Phi^{\{e\}}(E(m_0,\ldots,m_n))}=\ip{E(m_0)^{\smsh p^{n-1}}}=\ip{E(m_0)}\]
\end{proof}

\section{Consequences for Localizations of $N_\infty$-algebras}\label{sec5}
We begin this section by recalling a surprising theorem of Blumberg and Hill. In this section, $\tx{Map}_G(-,-)$ will denote the $G$-space of maps in the category of $G$-spaces.
\begin{theorem}\label{5.1}
$($\cite{BH}\textnormal{, Theorem 1.4}$)$ If $\mc O$ is an $N_\infty$-operad, and $R$ is an $\mc O$-algebra such that $R$ is cofree, then $R$ is equivalent (as an $\mc O$-algebra) to a genuine $G$-$E_\infty$ ring. 
\end{theorem}
\begin{proof}
$R\simeq F(EG_+,R)$, and since $R$ is an $\mc O$-algebra, $F(EG_+,R)$ is canonically a $\tx{Map}_G(EG,\mc O)$-algebra. Each $\tx{Map}_G(EG,\mc O_n)$ is a universal space for some family $\mc F$ of graph subgroups of $G\times\Sigma_n$, and if $\mc F'$ is any other such family, a map
\[E\mc F'\to \tx{Map}_G(EG,\mc O_n)\]
is the same thing  as a map
\[E\mc F'\times EG\to\mc O_n\]
But $E\mc F'\times EG\simeq EG$, hence there is always such a map, as $EG$ is initial in the category of such universal spaces. Therefore $\tx{Map}_G(EG,\mc O_n)$ is terminal, so it is an $E_G\Sigma_n$, and $\tx{Map}_G(EG,\mc O)$ is equivalent to the terminal $N_\infty$-operad.
\end{proof}
This is an extremely useful theorem: many genuine equivariant homotopy types come naturally as cofree spectra equipped with naive $E_\infty$-structures. The Lubin-Tate theories $E_{(k,\mathbb G)}$ with their actions by (subgroups of) the Morava stabilizer group, furnished by the Goerss-Hopkins-Miller theorem \cite{Rezk}, come to us this way, and similarly for various equivariant forms of $TMF$. For example, the $C_2$-spectrum $Tmf_1(3)$ studied by Hill and Meier \cite{HM}. These cofree theories $E$ therefore come equipped with canonical maps of genuine commutative ring spectra
\[N^TE\to E\]
for finite $G$-sets $T$. These maps play an essential role in computations involving the above spectra, see for example Section 6 of \cite{Hahn} in the $E_{(k,\mb G)}$-case.

We give a series of generalizations of this result that concern $H$-cofree $G$-spectra and induced localizations. This is a natural direction of generalization as $F(EG_+,-)$ is simply the induced Bousfield localization functor $L_{G_+}(-)$. Let $E\in Sp^G$ and let $\un{\mc Z_E}$ denote the nonunital symmetric monoidal coefficient system (as in \cite{HH}) of $E$-acyclics. That is, $\un{\mc Z_E}$ is the contravariant (pseudo-)functor from the orbit category $\mc O_G$ to nonunital symmetric monoidal categories with values
\[\un{\mc Z_E}(G/H)=\mc Z_{i^G_HE}\]
We now recall the following theorem of Hill-Hopkins and Gutierrez-White:
\begin{theorem}\label{5.2}
\cite{HH}\cite{GW} Let $\mc O$ be an $N_\infty$-operad for the group $G$, and $E\in Sp^G$. Then $L_E(-)$ preserves $\mc O$-algebras if and only if for all $K\subset H\subset G$ such that $H/K$ is an admissible $H$-set of $\mc O$,
\[N^{H/K}(\un{\mc Z_E}(G/H))\subset \un{\mc Z_E}(G/H)\]
\end{theorem}
\begin{proposition}\label{5.3}
Let $\mc O$ be an $N_\infty$ operad for the group $G$. If $H\subset G$ and $E\in Sp^H$ is such that $L_E(-)$ preserves $i^G_H\mc O$-algebras, then $L_{G_+\smsh_HE}(-)$ preserves $\mc O$-algebras.
\end{proposition}
\begin{proof}
Let $K'\subset K\subset G$ be such that $K/K'$ is an admissible $K$-set for $\mc O$, then we must show that
\[N^{K/K'}\un{\mc Z_{G_+\smsh_HE}}(G/K)\subset\un{\mc Z_{G_+\smsh_HE}}(G/K) \]
The double coset formula states
\[i^G_KG_+\smsh_HE=\bigvee\limits_{[g]\in K\backslash G/H}K_+\smsh_{K\cap {^gH}}i^{^gH}_{K\cap ^gH}({^gE})\]
hence $Z\in \un{\mc Z_{G_+\smsh_HE}}(G/K)\iff i^K_{K\cap ^gH}Z\in \mc Z_{i^{^gH}_{K\cap ^gH}(^gE)}$ for all $g\in G$. We therefore assume that $i^K_{K\cap ^gH}Z\in \mc Z_{i^{^gH}_{K\cap ^gH}(^gE)}$, and we must show that $i^K_{K\cap ^gH}N^{K/K'}(Z)\in\mc Z_{i^{^gH}_{K\cap ^gH}(^gE)}$, but we have
\[i^K_{K\cap ^gH}N^{K/K'}(Z)=\bigwedge\limits_{[h]\in (K\cap ^gH)\backslash K/K'}N^{K\cap ^gH}_{(K\cap ^gH)\cap ^h(K')}i^{^h(K')}_{(K\cap ^gH)\cap ^h(K')}{(^h(i^K_{K'}Z))}\]
This smash product is in $\mc Z_{i^{^gH}_{K\cap ^gH}(^gE)}$ if any of its factors is, hence we may take $h=e$ so that it suffices to show that
\[N^{K\cap ^gH}_{(K\cap ^gH)\cap K'}i^{K'}_{(K\cap ^gH)\cap K'}(i^K_{K'}Z)=N^{(K\cap ^gH)/((K\cap ^gH)\cap K')}(i^K_{K\cap ^gH}Z)\in \mc Z_{i^{^gH}_{K\cap ^gH}(^gE)}\]
Since $\mc O$ admits $K/K'$, $\mc O$ admits $(K\cap ^gH)/((K\cap ^gH)\cap K')$ since the admissible sets for $\mc O$ are closed under restriction in this way. If we knew then that $L_{^gE}$ preserves $i^G_{^gH}\mc O$-algebras, \hyperref[5.2]{5.2} would guarantee that $\mc Z_{i^{^gH}_{K\cap ^gH}(^gE)}$ is closed under this norm.

The fact that
\[L_E(-)\tx{ preserves }i^G_H\mc O\tx{-algebras}\implies L_{^gE}(-)\tx{ preserves }i^G_{^gH}\mc O\tx{-algebras}\]
follows from the fact that the admissible sets for $\mc O$ are closed under conjugacy, along with the observations
\[N^{K/K'}Z\smsh i^H_KE\simeq*\iff N^{^gK/^g(K')}(^gZ)\smsh i^{^gH}_{^gK}(^gE)\simeq *\]
\[Z\smsh i^H_KE\simeq*\iff ^gZ\cap i^{^gH}_{^gK}(^gE)\simeq*\]
which follow from the fact that $^g(-):Sp^H\to Sp^{^gH}$ is a symmetric monoidal equivalence of categories.
\end{proof}
\begin{corollary}\label{5.4}
If $H\subset G$, and $E\in Sp^H$ is such that $L_E(-)$ preserves $H$-commutative rings, then $L_{G_+\smsh_HE}(-)$ preserves $G$-commutative rings.
\end{corollary}
As in \hyperref[5.1]{5.1}, we will see that the situation is actually better than this: induced localizations automatically upgrade the available norms for an $N_\infty$-algebra, and we make this precise using the results of Section \hyperref[sec3]{3}. We may give $E\mc F_H$ the trivial $\Sigma_n$-action, and as such it becomes the universal $G\times\Sigma_n$-space for the family
\[\mc F_{H\times\Sigma_n}:=\{\Lambda\subset G\times\Sigma_n\st \tx{pr}_1(\Lambda)\in\mc F_H\}\]
where $\tx{pr}_1:G\times\Sigma_n\to G$ is the projection onto the first factor. It is easy to check that if $X$ is any $G\times\Sigma_n$-space, then if we give the $G$-space $\tx{Map}_G(E\mc F_H,X)$ a $G\times\Sigma_n$-action by postcomposing with the action of $\Sigma_n$ on $X$, this is isomorphic as a $G\times\Sigma_n$-space to $\tx{Map}_{G\times\Sigma_n}(E\mc F_H,X)$, where $E\mc F_H$ has a trivial $\Sigma_n$ action as above. If $\mc O$ is any $N_\infty$-operad for the group $G$, it follows that $\tx{Map}_G(E\mc F_H,\mc O)$ is as well. Moreover, if $R$ is an algebra over $\mc O$, then $F(E{\mc {F}_H}_+,R)$ is an algebra over $\tx{Map}_G(E\mc F_H,\mc O)$. \hyperref[3.7]{3.7} and \hyperref[5.3]{5.3} together give: 
\begin{corollary}\label{5.5}
If $R\in Sp^G$, $E\in Sp^H$, and $\mc O$ is an $N_\infty$ operad for the group $G$ such that $R$ is an $\mc O$-algebra and $L_E$ preserves $i^G_H\mc O$-algebras, then $L_{G_+\smsh_HE}(R)$ is a $\mathrm{Map}_G(E{\mc F_{H}},\mc O)$-algebra. 
\end{corollary}
In the situation of the corollary, we find that $R$ acquires \ti{more} norms after localizing at $G_+\smsh_HE$ since  the collection of admissible sets for $\tx{Map}_G(E{\mc F_{H}},\mc O)$ contains that of $\mc O$. We determine now exactly which new norms it acquires. If $\mc O$ is any $N_\infty$-operad for the group $H$, then the coinduced operad $F_H(G,\mc O)$ is an $N_\infty$-operad for the group $G$ (\cite{BH}, 6.14), and we have the following:
\begin{proposition}\label{5.6}
$\mathrm{Map}_G(E\mc F_H,\mc O)\simeq F_H(G,i^{G}_{H}\mc O)$ as $N_\infty$-operads.
\end{proposition}
\begin{proof}
We have a zig zag of maps of operads
\[
\btz
\tx{Map}_G(E\mc F_H,\mc O)&\tx{Map}_G(E\mc F_H,\mc O)\times F_H(G,i^{G}_{H}\mc O)\arrow[l]\arrow[r]&F_H(G,i^{G}_{H}\mc O)
\etz
\]
given by the projection maps. It follows that if, for all $n\ge0$, $\tx{Map}_G(E\mc F_H,\mc O)_n$ and $F_H(G,i^{G}_{H}\mc O)_n$ are universal $G\times\Sigma_n$-spaces for the same family of subgroups, then both projections are equivalences.

Let $\mc U_n$ be the category of universal $(G\times\Sigma_n)$-spaces $E\mc F$ for $\mc F$ a family of graph subgroups of $G\times\Sigma_n$. It is an immediate consequence of Elmendorf's theorem that $Ho(\mc U_n)$ is equivalent to the poset of families of graph subgroups of $G\times\Sigma_n$, via inclusion. Therefore, if $E\in\mc U_n$, then $E=E\mc F$ is a universal $G\times\Sigma_n$-space for the family of subgroups 
\[\mc F=\bigcup\limits_{\substack{\mc F'\st \\\exists E\mc F'\to E}}\mc F'\]
given by the union of families $\mc F'$ having the property that there is a $G\times\Sigma_n$-equivariant map $E\mc F'\to E$. For $\tx{Map}_G(E\mc F_H,\mc O_n)$, by adjunction, there is such a map if and only if there is a map
\[E\mc F\times E\mc F_H\to\mc O_n\]
Since $E\mc F\times E\mc F_H\simeq E(\mc F\cap \mc F_{H\times\Sigma_n})$, this happens if and only if
\[\mc F\cap \mc F_{H\times\Sigma_n}\subset\mc F_{\mc O_n}\]
One may show that $F_H(G,i^{G\times\Sigma_n}_{H\times\Sigma_n}\mc O_n)\cong F_{H\times\Sigma_n}(G\times\Sigma_n,i^{G\times\Sigma_n}_{H\times\Sigma_n}\mc O_n)$ so that there is a $G\times\Sigma_n$-map
\[E\mc F\to F_H(G,i^{G\times\Sigma_n}_{H\times\Sigma_n}\mc O_n)\]
if and only if there is a map
\[i^{G\times\Sigma_n}_{H\times\Sigma_n}E\mc F\to i^{G\times\Sigma_n}_{H\times\Sigma_n}\mc O_n\]
by adjunction. One checks easily that these are the following universal $H\times\Sigma_n$-spaces
\[i^{G\times\Sigma_n}_{H\times\Sigma_n}E\mc F=E(\Gamma\subset H\times\Sigma_n\st \Gamma\in\mc F)\]
\[i^{G\times\Sigma_n}_{H\times\Sigma_n}\mc O_n=E(\Gamma\subset H\times\Sigma_n\st \Gamma\in\mc O_n)\]
Hence the map above exists if and only if
\[\{\Gamma\subset H\times\Sigma_n\st \Gamma\in\mc F\}\subset\mc F_{\mc O_n}\]
Since $\mc F_{\mc O_n}$ is a family and in particular closed under subconjugates, this happens if and only if
\[\mc F\cap\{\Gamma\subset G\times\Sigma_n\st \Gamma\tx{ is subconjugate to }H\times\Sigma_n\}\subset\mc F_{\mc O_n}\]
It therefore suffices to observe that
\[\mc F_{H\times\Sigma_n}=\{\Gamma\subset G\times\Sigma_n\st \Gamma\tx{ is subconjugate to }H\times\Sigma_n\}\]
\end{proof}

\begin{corollary}\label{5.7}
For any $K\subset G$, a $K$-set $T$ is admissible for $\mathrm{Map}_G(E\mc F_H,\mc O)$ if and only if for all $g\in G$, $i^{gKg^{-1}}_{H\cap gKg^{-1}}{^gT}$ is admissible for $\mc O$. In particular, if $i^G_H\mc O$ is genuine $H$-$E_\infty$, then $\mathrm{Map}_G(E\mc F_H,\mc O)$ is genuine $G$-$E_\infty$.
\end{corollary}
\begin{proof}
It is clear that if $K\subset H$, and $T$ is a $K$-set, then $i^{G}_H\mc O$ admits $T$ iff $\mc O$ admits $T$. Now we apply the previous proposition and (\cite{BH}, 6.16).
\end{proof}
The following is the most direct generalization of \hyperref[5.1]{5.1} above, in the case where $i^G_H\mc O$ is genuine $H$-$E_\infty$.
\begin{corollary}\label{5.8}
Let $R\in Sp^G$ be an algebra over an $N_\infty$-operad $\mc O$ such that $i^G_H\mc O$ is a genuine $H$-$E_\infty$-operad, and let $E\in Sp^H$. If $L_E(-)$ preserves $H$-commutative rings, then $L_{G_+\smsh_HE}(R)$ is a $G$-commutative ring. In particular, if $R$ is an $\mc O$-algebra such that $i^G_H\mc O$ is genuine $H$-$E_\infty$, then $F(E{\mc F_H}_+,R)$ is a $G$-commutative ring.
\end{corollary}
\begin{proof}
Note that $F(E{\mc F_H}_+,R)\simeq L_{G/H_+}(R)\simeq L_{G_+\smsh_HS^0}(R)$, so that the second assertion follows from the first. For the first assertion, we simply combine \hyperref[5.5]{5.5} and \hyperref[5.7]{5.7}.
\end{proof}
\begin{example}\label{5.9}
$L_{E_G(m)}$ sends $\mc O$-algebras to $G$-commutative rings for all $n$ and $m$.
\end{example}

\section{Restriction of Idempotents along a Quasi-Galois Extension}\label{sec6}
We digress from categories of $G$-spectra to highlight the extent to which the results of Section \hyperref[sec3.3]{3.3} may be generalized to other settings in which Bousfield localization is possible. This is motivated by the following theorem of Balmer, Dell'Ambroglio, and Sanders:
\begin{theorem}\label{6.1}
$($\cite{BDS}\textnormal{, Theorem 1.1}$)$ For $H\subset G$ a subgroup, there is an equivalence of tt-categories
\[Ho(Sp^H)\simeq \mathrm{Mod}_{Ho(Sp^G)}(F(G/H_+,S^0))\]
where the latter is the category of modules in $Ho(Sp^G)$ over the ring spectrum $F(G/H_+,S^0)$. Under this equivalence, the functor $i^G_H(-)$ corresponds to extension of scalars along the unit map $S^0\to F(G/H_+,S^0)$, and $G_+\smsh_H(-)\simeq F_H(G_+,-)$ corresponds to restriction of scalars.
\end{theorem}
Mathew, Noel, and Naumann upgraded this to an equivalence of symmetric monoidal $\infty$-categories \cite{MNN}
\[Sp^H\simeq \tx{Mod}_{Sp^G}(F(G/H_+,S^0))\]
and studied the extent to which a commutative algebra $A$ in a presentable, symmetric monoidal stable $\infty$-category $(\mc C,\otimes,\mathds 1)$ exhibited categorical properties similar to those seen in equivariant homotopy theory with $A=F(G/H_+,S^0)$. In our context, the analogy suggests that perhaps a smashing $A$-module $M$ will pull back to an object in $\mc C$ whose Bousfield localization functor becomes smashing after completion at $A$. That is, if $\eta:\mathds 1\to A$ is the unit map of $A$, if $M\in \tx{Mod}_{\mc C}(A)$ determines a smashing localization in $\tx{Mod}_{\mc C}(A)$, following \hyperref[3.20]{3.20}, we expect a formula
\[L_{\eta^*M}(-)=L_A(L_{\eta^*M}(\mathds{1})\otimes -)\]
in $\mc C$. However, \hyperref[3.24]{3.24} tells us that, even in the motivating example $S^0\to F(G/H_+,S^0)$, we need $H$ to be normal for such a formula to hold. Hence we are led to ask that $\eta$ be a quasi-Galois extension.

\subsection{Background on stable $\infty$-categories and quasi-Galois extensions.}\label{sec6.1} We review what is needed to establish the desired localization formulae for a quasi-Galois extension. We use the language of $\infty$-categories following \cite{HTT} and closely follow the discussion in Section 1 of \cite{MNN}, where more detail can be found. In all that follows, we will let $(\mc C,\otimes,\mathds{1})$ be a presentable, symmetric monoidal stable $\infty$-category in which $-\otimes-$ commutes with colimits in colimits in each variable.

\begin{definition}\label{6.2}
Let $M\in\mc C$. We let $\mc Z_{M}$ be the full subcategory of $\mc C$ consisting of those $Z\in\mc C$ such that $Z\otimes M\simeq *$. We let $\mc L_M$ denote the full subcategory of $\mc C$ consisting of those $Y\in \mc C$ such that the space $\tx{Map}_{\mc C}(Z,Y)$ is contractible for all $Z\in\mc Z_{M}$.
\end{definition}
It follows formally from (\cite{HTT}, Section 5.5) that $\mc L_M$ is a presentable stable $\infty$-category, and the inclusion $\mc L_M\inj\mc C$ admits a left adjoint, $L_M(-)$. Moreover, by (\cite{HA}, 2.2.1.9), $\mc L_M$ inherits the structure of a symmetric monoidal $\infty$-category so that $L_M:\mc C\to\mc L_M$ is symmetric monoidal. The tensor product in $\mc L_M$ is then necessarily given by the formula
\[L_M(X)\hat{\otimes} L_M(Y):=L_M(X\otimes Y)\]
With this in place, the discussion in Section \hyperref[sec2]{2} may be repeated in this setting \ti{mutatis mutandis}. In particular, we may use smashing localizations and tensor idempotents interchangeably (see \cite{HA}, Section 6.3 or \cite{GGN}, Section 3 for more details).

Suppose now we have an object $A\in\tx{CAlg}(\mc C)$ - this induces an Ind-Res adjunction
\[
\btz
\mc C=\tx{Mod}_{\mc C}(\mathds{1})\arrow[d,shift right=1ex,"\eta_*"']\\
\tx{Mod}_{\mc C}(A)\arrow[u,shift right=1ex,"\eta^*"']
\etz
\]
$\tx{Mod}_{\mc C}(A)$ is a presentable, symmetric monoidal stable $\infty$-category, $\eta_*$ is a symmetric monoidal functor, and the adjunction $\eta_*\dashv \eta^*$ satisfies the projection formula
\[N\otimes \eta^*(M)\simeq \eta^*(\eta_*N\otimes_AM)\]
(see \cite{MNN}, Section 5.2). Under the assumption that $A$ is dualizable in $\mc C$, \cite{MNN} deduce the following description of $\mc L_A$.\\

\begin{theorem}\label{6.3}
If $A$ is dualizable in $\mc C$, the functor $\eta_*$ descends to an equivalence of symmetric monoidal $\infty$-categories
\[
\btz
\mc L_A\simeq \mathrm{Tot}(\mathrm{Mod}_{\mc C}(A)\arrow[r,shift left=0.5ex]\arrow[r,shift right=0.5ex]&\mathrm{Mod}_{\mc C}(A\otimes A)\arrow[r,shift left=0.5ex]\arrow[r]\arrow[r,shift right=0.5ex]&\cdots)
\etz
\]
\end{theorem}

In our motivating example of $A=F(G/H_+,S^0)\in\tx{CAlg}(Sp)$ for $H\lhd G$, the double coset formula allows us to identify the simplicial object on the right hand side as the cobar complex computing $(Sp^H)^{h(G/H)}$. This generalizes to the following situation:

\begin{definition}\label{6.4} Let $G$ be a finite group, $R\in\tx{CAlg}(\mc C)$, and $A\in \tx{Fun}(BG,\tx{CAlg}(\mc C)_{R/})$. Consider the diagram in $\tx{CAlg}(\mc C)$.
\[
\btz
R\arrow[r]\arrow[d]&A\arrow[ddr,"\Delta^{tw}",bend left]\arrow[d]\\
A\arrow[r]\arrow[rrd,"\Delta"',bend right]&A\otimes_RA\arrow[dr,dashed,"\phi"]\\
&&\prod\limits_{g\in G}A
\etz
\]
where $\pi_g\circ\Delta^{tw}=g:A\to A$. We say that $R\to A$ is a \tb{quasi-Galois extension} if $\phi$ is an equivalence.
\end{definition}

\begin{remark}\label{6.5}
If we required additionally that the morphism $R\to A^{hG}$ be an equivalence, this would be the usual definition of a Galois extension, due to Rognes \cite{Rognes}. This terminology is used in \cite{Quasi} where quasi-Galois extensions are studied in a tt-geometry context.
\end{remark}

As before, we will take $R=\mathds 1$, and we record the following immediate consequence of \hyperref[6.3]{6.3}:

\begin{corollary}\label{6.6}
If $A$ is dualizable in $\mc C$, and $\eta:\mathds 1\to A$ is a quasi-Galois extension, the functor $\eta_*$ descends to an equivalence of symmetric monoidal $\infty$-categories
\[
\mc L_A\simeq (\mathrm{Mod}_{\mc C}(A))^{hG}
\]
\end{corollary}

\begin{remark}\label{6.7}
If $\eta$ were a Galois extension, the dualizability condition on $A$ would be automatic (\cite{Rognes}, 6.2.1).
\end{remark}

When $\eta$ is a quasi-Galois extension, the projection formula gives the following decomposition, of which the double-coset formula for $i^G_H(G_+\smsh_H-)$ is a special case.
\begin{lemma}\label{6.8}
For $M\in\mathrm{Mod}_{\mc C}(A)$,
\[\eta_*\eta^*M=\bigoplus\limits_{g\in G}{^gM}\]
\end{lemma}

\subsection{Smashing $A$-modules}\label{sec6.2}
Our desired localization formulae are of the form
\[L(-)=L_A(L(\mathds 1)\otimes-)\]
By definition of the symmetric monoidal structure in $\mc L_A$, producing a localization functor $L(-)$ on $\mc C$ given by such a formula is equivalent to producing a smashing localization in $\mc L_A$. Corollary \hyperref[6.6]{6.6} tells us that smashing localizations in $\mc L_A$ are the same thing as smashing localizations in $(\tx{Mod}_{\mc C}(A))^{hG}$. This allows us to produce smashing localizations in $\mc L_A$ from smashing localizations in $\tx{Mod}_{\mc C}(A)$ via norm functors.

\begin{construction}\label{6.9}
Let $(\mc D,\otimes,\mathds 1)$ be a presentable, symmetric monoidal $\infty$-category with $G$-action (e.g. $\tx{Mod}_{\mc C}(A)$ as above). There is a symmetric monoidal functor
\[N:\mc D\to \mc D^{hG}\]
such that the composite
\[\mc D\x{N}\mc D^{hG}\inj D\]
is given by the functor
\[M\mapsto\bigotimes\limits_{g\in G}{^gM}\]
\end{construction}

\begin{remark}\label{6.10}
There is a right adjoint 
\[\tx{Fun}(BG,\tx{SymMon}\infty\tx{-Cat})\to G\tx{-SymMon}\infty\tx{-Cat}\]
to the forgetful functor which sends $(\mc D,\otimes,\mathds 1)$, a presentable, symmetric monoidal $\infty$-category with $G$-action, to the $G$-symmetric monoidal $\infty$-category $\underline{\mc D}(G/H)=\mc D^{hH}$, with norm map $\underline{\mc D}(G/e)\to \underline{\mc D}(G/G)$ as in \hyperref[6.9]{6.9}. This is a higher algebra analog of the functor that sends a commutative ring with $G$-action to its fixed-point Tambara functor. An account of this construction is to appear in \cite{PHT}.

This construction appears in the context of the symmetric monoidal $G$-categories of Guillou, May, Merling, and Osorno (see 3.7 of \cite{GMMO}), the normed symmetric monoidal categories of Rubin (see 3.7 of \cite{Rubin}), and the symmetric monoidal mackey functors of Hill-Hopkins (see 2.6 \cite{HH}).
\end{remark}

By use of $N$, we may therefore send an idempotent $e$ in $\tx{Mod}_{\mc C}(A)$ to an idempotent $N(e)$ in $\mc L_A$. This determines some smashing localization in $\mc L_A$, and using \hyperref[6.8]{6.8}, we may identify its corresponding Bousfield class in terms of that of $e$. We have the following:

\begin{theorem}\label{6.11}
Suppose $(\mc C,\otimes,\mathds 1)$ and $A$ are as above. In particular, assume $\eta:\mathds 1\to A$ is a quasi-Galois extension and $A$ is dualizable in $\mc C$. If $M\in\mathrm{Mod}_{\mc C}(A)$ is smashing, then we have the formula in $\mc C$:
\[L_{\eta^*M}(-)=L_A(L_{\eta^*M}(\mathds 1)\otimes -)\]
\end{theorem}
\begin{proof}
For $X\in\mc C$, the composite
\[X\to L_{\eta^*M}(\mathds 1)\otimes X\to L_A(L_{\eta^*M}(\mathds 1)\otimes X)\]
becomes an equivalence after applying $-\otimes \eta^*M$. This is clear for the first map, and for the second map, for any $Y\in \mc C$, we have a commutative diagram
\[
\btz
Y\otimes \eta^*M\arrow[r]\arrow[d,"\simeq"]&L_A(Y)\otimes\eta^*M\arrow[d,"\simeq"]\\
\eta^*(\eta_*Y\otimes_AM)\arrow[r]&\eta^*(\eta_*(L_A(Y))\otimes_AM)
\etz
\]
by the projection formula. The bottom arrow is an equivalence because $\eta_*Y\to\eta_*(L_A(Y))$ is an equivalence by definition. It suffices now to show that $L_A(L_{\eta^*M}(\mathds 1)\otimes X)$ is $\eta^*M$-local in $\mc C$. 

Suppose we knew that $L_A(\eta^*M)$ determined a smashing Bousfield class in $\mc L_A$. Then if $Z\in\mc Z_{\eta^*M}$, we have $L_A(Z)\hat{\otimes}L_A(\eta^*M)\simeq*$, and
\[
\tx{Map}_{\mc C}(Z,L_A(L_{\eta^*M}(\mathds 1)\otimes X))\simeq\tx{Map}_{\mc L_A}(L_A(Z),L_A(L_{\eta^*M}(\mathds 1))\hat{\otimes} L_A(X))
\]
Since $L_A(\eta^*M)$-locals form a $\hat{\otimes}$-ideal in $\mc L_A$ by assumption, it would therefore suffice to show that $L_A(L_{\eta^*M}(\mathds 1))$ is $L_A(\eta^*M)$-local in $\mc L_A$. If $L_A(Z')\in\mc L_A$ is an $L_A(\eta^*M)$-$\hat{\otimes}$-acyclic, then
\[
\tx{Map}_{\mc L_A}(L_A(Z'),L_A(L_{\eta^*M}(\mathds 1)))\simeq\tx{Map}_{\mc C}(Z',L_{\eta^*M}(\mathds 1))\simeq*\]
The first equivalence is by adjunction and the fact that $\mc L_{\eta^*M}\subset\mc L_{A}$, as $\mc Z_{A}\subset\mc Z_{\eta^*M}$ by the projection formula. For the second, $Z'\otimes \eta^*M$ is $A$-local, as an $A$-module, hence 
\[Z'\otimes \eta^*M\simeq L_A(Z'\otimes\eta^*M)\simeq L_A(Z')\hat{\otimes}L_A(\eta^*M)\simeq*\]

We now show that $L_A(\eta^*M)$ determines a smashing Bousfield class in $\mc L_A$. Let $e_M,f_M\in \tx{Mod}_{\mc C}(A)$ denote the left and right idempotents corresponding to $M$, respectively. If $N(-)$ is the functor in \hyperref[6.9]{6.9}, we have that 
\[\ker(-\hat{\otimes}\tx{cofib}(N(e_M)\to\mathds 1))\]
is a smashing ideal in $\mc L_A$. It suffices to show that it coincides with $\ker(-\hat{\otimes}L_A(\eta^*M))$. Since $\eta_*$ is conservative on $\mc L_A$, we have
\aln{
Z\hat{\otimes}\tx{cofib}(N(e_M)\to\mathds 1)\simeq*&\iff \eta_*(Z\hat{\otimes}\tx{cofib}(N(e_M)\to\mathds 1))\simeq*\\
&\iff \eta_*(Z)\otimes_A\tx{cofib}(\eta_*N(e_M)\to\mathds 1)\simeq*\\
&\iff \eta_*(Z)\otimes_A\tx{cofib}(\bigotimes\limits_{g\in G}{^ge_M}\to\mathds 1)\simeq*\\
&\iff \eta_*(Z)\otimes_A\tx{cofib}(e_{\bigoplus\limits_{g\in G}{^gM}}\to \mathds 1)\simeq *\\
&\iff  \eta_*(Z)\otimes_A f_{\bigoplus\limits_{g\in G}{^gM}}\simeq*\\
&\iff  \eta_*(Z)\otimes_A \bigoplus\limits_{g\in G}{^gM}\simeq*\\
&\iff  \eta_*(Z)\otimes_A \eta_*\eta^*M\simeq*\\
&\iff \eta_*(Z\otimes \eta^*M)\simeq*\\
&\iff Z\hat{\otimes}\eta^*M\simeq*
}
The third equivalence is by definition of $N(-)$, the fourth and sixth follow as in \hyperref[2.12]{2.12}, and the seventh is \hyperref[6.8]{6.8}.
\end{proof}

\begin{remark}\label{6.12}
When $\mc C=Sp^G$ and $A=F(G/H_+,S^0)$ for $H\lhd G$, this recovers \hyperref[3.20]{3.20}. Moreover, the description of $L_A(-)$ in this case may be generalized: it follows from (\cite{MNN}, Proposition 2.21) that in the situation of \hyperref[6.11]{6.11}, we have the formula
\[L_A(Y)\simeq (Y\otimes A)^{hG}\simeq F(\mb D(A),Y)^{hG}\simeq F(\mb D(A)_{hG},Y)\]
and so the formula in \hyperref[6.11]{6.11} may be made more explicit:
\[L_{\eta^*M}(X)\simeq F(\mb D(A)_{hG},L_{\eta^*M}(\mathds 1)\otimes X)\] 
\end{remark}

\begin{remark}\label{6.13}
We may derive an analogous formula for a general quasi-Galois extension $\eta:R\to A$ in $\mc C$ such that $A$ is dualizable in $\tx{Mod}_{\mc C}(R)$: a smashing $R$-linear $A$-module $M\in\tx{Mod}_{\tx{Mod}_{\mc C}(R)}(A)$ determines a smashing localization in the category of $A$-locals in $\tx{Mod}_{\mc C}(R)$ corresponding to the Bousfield class of $\eta^*M$. The same proofs work since $R$ is the unit in $\tx{Mod}_{\mc C}(R)$.
\end{remark}

We have also in this setting a necessary and sufficient condition for $\eta^*M$ to be smashing in $\mc C$, i.e. a generalization of \hyperref[3.19]{3.19}. In \cite{MNN}, the role of the geometric fixed points functor is generalized to this setting as follows: consider the cofiber sequence
\[\mb D(A)\x{\mb D(\eta)}\mathds 1\x{a_A}C(A)\]
Define
\[U_A:=\colim(\mathds 1\x{a_A}C(A)\x{a_A}C(A)^{\otimes 2}\x{a_A}\cdots)\]
Then $U_A$ is a right idempotent in $\mc C$, and for any $X\in\mc C$, there is a homotopy pullback square
\[
\btz
X\arrow[r]\arrow[d]&X\otimes U_A\arrow[d]\\
L_A(X)\arrow[r]&L_A(X)\otimes U_A
\etz
\]

\begin{proposition}\label{6.14}
In the situation of \hyperref[6.11]{6.11}, $\eta^*M$ is smashing in $\mc C$ if and only if $L_{\eta^*M}(\mathds 1)\otimes U_A\simeq*$.
\end{proposition}
\begin{proof}
If $\eta^*M$ is smashing, then $L_{\eta^*M}(\mathds 1)\otimes U_A\simeq L_{\eta^*M}(U_A)$, but $\mc Z_A\subset \mc Z_{\eta^*M}$, and in the cofiber sequence
\[\mb D(A)\otimes A\to A\to C(A)\otimes A\]
the first map splits via the map $A\to \mb D(A)\otimes A$ adjoint to the multiplication map $A\otimes A\to A$. Therefore the map $A\to C(A)\otimes A$ is null, and so
\[U_A\otimes A\simeq\colim(A\to C(A)\otimes A\to C(A)^{\otimes 2}\otimes A\to\cdots)\simeq*\]
Conversely, suppose $L_{\eta^*M}(\mathds 1)\otimes U_A\simeq*$, we will show that $L_{\eta^*M}(\mathds 1)\otimes X\in\mc L_{\eta^*M}$ for all $X\in \mc C$. As above, we have a pullback square
\[
\btz
L_{\eta^*M}(\mathds 1)\otimes X\arrow[r]\arrow[d]&L_{\eta^*M}(\mathds 1)\otimes X\otimes U_A\arrow[d]\\
L_A(L_{\eta^*M}(\mathds 1)\otimes X)\arrow[r]&L_A(L_{\eta^*M}(\mathds 1)\otimes X)\otimes U_A
\etz
\]
By \hyperref[6.11]{6.11}, we may identify the bottom row with the map $L_{\eta^*M}(X)\to L_{\eta^*M}(X)\otimes U_A$. By assumption, $L_{\eta^*M}(\mathds 1)\otimes U_A\simeq*$, so $L_{\eta^*M}(X)\otimes U_A$ is contractible as a module over $L_{\eta^*M}(\mathds 1)\otimes U_A$. Therefore the left hand arrow is an equivalence and the target is $\eta^*M$-local.
\end{proof}

\begin{corollary}\label{6.15}
In the situation of \hyperref[6.11]{6.11} $\eta^*M$ is smashing in $\mc C$ if and only if $L_{\eta^*M}(\mathds 1)$ is in the thick subcategory of $\mc C$ generated by $A$.
\end{corollary}
\begin{proof}
This follows immediately from \cite{MNN}, Theorem 4.19.
\end{proof}

\begin{example}\label{6.16}
Let $H\lhd G$ be a closed normal subgroup of finite index in a compact Lie group $G$. If $E\in Sp^H$ is smashing, then
\[L_{G_+\smsh_HE}(X)=F({E{\mc F}_H}_+,L_E(S^0)\smsh X)\]
for all $X\in Sp^G$.
\end{example}
\begin{proof}
In \cite{MNN}, pg. 29, it is noted that the the analog of \hyperref[4.1]{4.1} (and its $\infty$-categorical refinement) hold in this setting, and the formula now follows from \hyperref[6.12]{6.12}.
\end{proof}

\begin{example}\label{6.17}
In fact, our above arguments may be used to show the analogous version of 3.20 for induced localizations holds for any of the equivariant tt-categories studied in \cite{BDS}.
\end{example}

\bibliographystyle{plain}

\end{document}